\documentclass[11pt]{article}

\usepackage{amsmath,amssymb,amsthm,mathtools}
\usepackage{enumitem}
\usepackage{hyperref}
\usepackage{color}
\usepackage{float}
\usepackage{graphicx}
\usepackage{fontawesome5}
\usepackage[normalem]{ulem}
\usepackage{orcidlink}

\usepackage{todonotes}

\def\cC{\mathcal C}

\def\cK{\mathcal K}

\newtheorem{theorem}{Theorem}[section]
\newtheorem{proposition}[theorem]{Proposition}
\newtheorem{corollary}[theorem]{Corollary}
\newtheorem{lemma}[theorem]{Lemma}
\theoremstyle{definition}
\newtheorem{definition}[theorem]{Definition}
\newtheorem{example}[theorem]{Example}
\theoremstyle{remark}
\newtheorem{remark}[theorem]{Remark}

\newcommand{\dom}{\operatorname{dom}}

\newcommand{\Lip}{\operatorname{Lip}}

\newcommand{\R}{\mathbb{R}}
\newcommand{\Rp}{\mathbb{R}_{+}}

\newcommand{\N}{\mathbb{N}}
\newcommand{\Q}{\mathbb{Q}}

\newcommand{\intt}{\operatorname{int}}
\newcommand{\bd}{\operatorname{bd}}
\newcommand{\co}{\operatorname{co}}
\newcommand{\clco}{\operatorname{cl\,co}}
\newcommand{\cone}{\operatorname{cone}}
\newcommand{\dist}{\operatorname{dist}}

\newcommand{\cl}{\operatorname{cl}}

\title{Separation of Regular Co-Radiant Sets via Inner Approximating Sequences, Radial Profile Functions and Applications to Optimization}

\author{
Fernando Garc\'ia-Casta\~no\thanks{\texttt{fernando.gc@ua.es} ORCID: 0000-0002-8352-8235}
\and
Miguel \'Angel Melguizo-Padial\thanks{\texttt{ma.mp@ua.es} ORCID:0000-0003-0303-791X}
\\[1ex]
\small Department of Mathematics, University of Alicante, \\ \small 03690 San Vicente del Raspeig,  Alicante, Spain
}

\date{}

\begin{document}

\maketitle


\begin{abstract}
We develop a new framework for the analysis of co-radiant sets based on a new notion of regularity. This framework extends the classical theory beyond co-radiant sets admitting a norm-base and provides a unified setting for separation and optimization. We introduce two new tools: inner approximating sequences and the radial profile function. The former provides a sequential approach to regularity and is the key ingredient in the proof of a general separation theorem for regular co-radiant sets, extending the existing separation theory. The latter is defined for arbitrary co-radiant sets and yields new characterizations of the existence of norm-bases, together with sequence-based characterizations of regularity under a solidness assumption on the associated cone. It also gives rise to a scalar function leading to scalarization results for approximate efficiency without additional structural assumptions on the underlying co-radiant set. Finally, we provide a variational interpretation of the separation residual by establishing radial-depth and metric error bounds. These yield sufficient conditions for approximate radial feasibility in equilibrium problems and variational inequalities, and an exact penalization principle on the zero level set of the residual.
\end{abstract}


\medskip


\section{Introduction}

Approximate solutions are unavoidable in optimization and equilibrium
theory. In vector optimization, the first systematic notion of approximate
efficiency is usually attributed to Kutateladze \cite{Kutateladze1979}.
Further refinements were introduced by Loridan \cite{Loridan1984},
V\'alyi \cite{Valyi1985}, White \cite{White1986}, N\'emeth
\cite{Nemeth1986}, Helbig \cite{Helbig1992}, and Tanaka
\cite{Tanaka1995}. These notions differ in the geometry of the admissible
perturbations and in their stability as the approximation parameter tends
to zero. A unifying approach was proposed by Guti\'errez, Jim\'enez and
Novo \cite{Gutierrez2006a} through co-radiant perturbation
sets: a co-radiant set determines which improvements are relevant at a
given precision level, and several classical notions of $\varepsilon$-efficiency
are recovered as particular cases. Sayadi-bander et al.\
\cite{SayadiBander2017} later introduced a separation and scalarization
framework for co-radiant sets in finite-dimensional spaces with compact
bases, inspired by nonlinear cone separation, and this approach has
motivated subsequent developments.

A nonlinear separation theorem for co-radiant sets in normed spaces was
obtained in \cite{GarciaCastanoMelguizo2025JGO}, using the cone
separation theorem of \cite{GarciaCastanoMelguizoParzanese2023}.
This yielded scalarization results for approximate and properly
approximate efficient points under weaker assumptions than those
previously available. A central hypothesis in
\cite{GarciaCastanoMelguizo2025JGO}, however, was that one of the
co-radiant sets admits a norm-base a natural assumption, but a
restrictive one, as the examples in Section~2 below illustrate.

In this paper we develop a new framework for the analysis of co-radiant
sets based on a new notion of regularity, which extends the
classical theory beyond co-radiant sets admitting a norm-base and
provides a unified setting for separation and optimization theory. A
regular co-radiant set need not have a global norm-base; instead, it
admits inner conical reductions whose spherical sections carry the
norm-base structure needed for separation. Regularity therefore
preserves, along an approximating family, the radial geometry required by
the separation argument, even when it fails globally. This relaxation is
motivated by \cite[Theorem~4.9]{GarciaCastanoMelguizo2025JGO}, where
separation is obtained from an increasing sequence of co-radiant subsets
admitting norm-bases; the present paper makes this idea precise through
the notion of an inner approximating sequence, developed
independently of any separation statement.

To replace the norm-base assumption by regularity, we introduce two new tools: inner approximating sequences
and the radial profile function.

The former provides a sequential approach to regularity. We first study
inner approximations of cones and prove stability properties for the
strict separation property for cones (SSP). While outer approximation schemes
based on sequences of cones containing a fixed cone are well
established in vector optimization, and have been used successfully in
the study of efficient solutions and stability (see, e.g.,
\cite{Gong1995Density,Gong1994Connectedness} and the references therein), and related
approximation ideas also appear in vector equilibrium problems and
Ekeland-type variational principles via approximating families of cones
(see \cite{Hai2021Ekeland}), we propose here a complementary, inner approximation
framework, based on nested cones contained in the underlying structure.
This explicit construction appears to be new in the context of co-radiant sets.
Inner approximating sequences constitute the key ingredient in the proof
of a general separation theorem for regular co-radiant sets, which
unifies and extends the existing separation theory for co-radiant sets.

The latter, the radial profile function, is defined for an arbitrary
co-radiant set, with no structural assumptions, and yields new
characterizations of the existence of norm-bases, as well as
sequence-based conditions for regularity under a solidness assumption on
the associated cone. In particular, norm-base existence is characterized
by the boundedness of the radial profile function; to the best of our
knowledge, this is the first such characterization in the literature.
Moreover, the radial profile function gives rise to a scalar function
leading to scalarization results for approximate efficiency, valid for
every co-radiant set and without requiring additional structural
assumptions on the underlying set, in contrast with related scalarization
schemes (e.g.\ \cite{GutierrezJimenezNovo2006MMOR,GaoYangTeo2011,SayadiBander2017,ZhaoChenYang2015,GarciaCastanoMelguizo2025JGO}).

Given two suitable co-radiant sets \((\mathcal C,\mathcal K)\), we prove a nonlinear separation theorem of Bishop--Phelps type, separating \(\mathcal C\) from the radial regions generated by \(\mathcal K\): the set \(\mathcal C\) is contained in the zero set of the resulting scalar residual, while this zero set is in turn contained in a prescribed radial-depth region associated with \(\mathcal K\). This unifies and extends the separation theorem of \cite{GarciaCastanoMelguizo2025JGO} from the norm-base setting to the regular setting, thereby enlarging the class of admissible separating pairs \((\mathcal C,\mathcal K)\). In particular, the resulting framework includes pairs associated with the principal co-radiant perturbation schemes underlying the classical notions of approximate efficiency recalled above; see \cite{Gutierrez2006a}.

Finally, we give a variational reading of the separation theorem. Error bounds estimate the distance to a solution or feasible set in terms of a scalar measure of violation. Their origins go back to Hoffman-type estimates and their extensions \cite{Hoffman1952,Robinson1975,Pang1997}, and they are closely related to stability properties of feasible-set mappings and solution mappings, including metric regularity and metric subregularity; see, for instance, \cite{RockafellarWets1998,DontchevRockafellar2014,Ioffe2000,FabianHenrionKrugerOutrata2010,KrugerNgaiThera2010,KrugerLopezThera2018}. In the present setting, nonlinear separation generates a nonnegative scalar residual measuring the violation of the separating inequality. We show that this residual controls both the radial-depth deficit and the distance to its zero set. The first estimate yields certificates of approximate radial feasibility for equilibrium problems and, in particular, for variational inequalities, in the spirit of residual and gap-function approaches \cite{CharithaDuttaLuke2015,DinhPham2020}. The second estimate yields an exact penalty formulation for optimization problems constrained by the zero set of the residual, in line with the classical theory of exact penalties and merit functions \cite{Burke1991,Ye2012,LeThiPhamDinhNgai2012}.

The paper is organized as follows. Section~2 fixes the notation and
discusses co-radiant sets, norm-bases, and the SSP. Section~3 studies inner approximations of cones and the
stability of the SSP under such approximations. Section~4 introduces
regular co-radiant sets and characterizes regularity via inner
approximating sequences. Section~5 develops the radial profile function,
its characterizations of norm-base existence and regularity, and its
scalarization consequences for approximate efficiency. Section~6 proves
the separation theorem for regular co-radiant sets. Section~7 derives the
radial-depth and metric error bounds induced by the separation residual,
and applies them to approximate radial feasibility in equilibrium
problems and variational inequalities, and to exact penalization.

\section{Preliminaries}
\label{sec:preliminaries}

Throughout the paper, \(X\) is a real normed space, \(X^*\) is its topological
dual, \(\|\cdot\|\) denotes the norm of \(X\), and \(\|\cdot\|_*\) the dual norm on
\(X^*\). We write
\[
B_X:=\{x\in X:\|x\|\le 1\}, \ \ 
B_X^\circ:=\{x\in X:\|x\|<1\}, \ \ S_X:=\{x\in X:\|x\|=1\}.
 \] 
for the closed unit ball, the open unit ball, and the unit sphere of \(X\),
respectively. The same notation is used in any normed space.

For \(t\in\R\), let \(t_+:=\max\{t,0\}\), and write
\(\R_+:=[0,+\infty)\) and \(\R_{++}:=(0,+\infty)\).

Let \(A\subset X\). We denote by \(\cl A\), \(\intt A\), \(\bd A\), \(A^c\), and
\(\co A\) its closure, interior, boundary, complement, and convex hull. For
\(r\ge0\), set \(A_{\ge r}:=\{a\in A:\|a\|\ge r\}\). For \(x\in X\), set
\(\mathrm{dist}(x,A):=\inf_{a\in A}\|x-a\|\) the distance from $x$ to $A$.

 A function \(f:A\to\R\) is Lipschitz continuous if
\(|f(x)-f(y)|\le L\|x-y\|\) for all \(x,y\in A\) and some \(L\ge0\). Its
Lipschitz constant is
\[
\Lip(f):=
\sup_{\substack{x,y\in A\\ x\neq y}}
\frac{|f(x)-f(y)|}{\|x-y\|}
=
\inf\{L\ge0:\ |f(x)-f(y)|\le L\|x-y\|\ \forall x,y\in A\}.
\]
Thus \(f\) is Lipschitz continuous if and only if \(\Lip(f)<+\infty\).

\subsection*{Cones and the SSP}
\label{subsec:cones-and-ssp}

A nonempty set \(C\subset X\) is a cone if \(\alpha C\subset C\) for every
\(\alpha\in\R_+\). Unless otherwise stated, cones are nontrivial:
\(\{0_X\}\subsetneq C\subset X\). A cone is pointed if
\(C\cap(-C)=\{0_X\}\), and solid if \(\intt C\neq\varnothing\). For
\(A\subset X\), set \(\cone(A):=\{\lambda a:\lambda\ge0,\ a\in A\}\).

For a cone \(C\subset X\), set
\[
\begin{aligned}
C^*&:=\{f\in X^*: f(x)\ge0\ \forall x\in C\},\\
C^\#&:=\{f\in C^*: f(x)>0\ \forall x\in C\setminus\{0_X\}\},\\
C^{a\#}_+&:=\{(f,\alpha)\in C^\#\times\R_{++}: f(x)-\alpha\|x\|>0\ \forall x\in C\setminus\{0_X\}\}.
\end{aligned}
\]
\begin{definition}
\label{def:ssp}
Let \(C,K\subset X\) be cones. The pair \((C,K)\) verifies the strict separation
property, briefly SSP, if
\[
0_X\notin
\cl\Bigl(
\co(C\cap S_X)-
\co\bigl((\bd K\cap S_X)\cup\{0_X\}\bigr)
\Bigr).
\]
\end{definition}

We next recall an analytic form of the SSP which will be used throughout the paper.

\begin{theorem}[{\cite[Theorem 3.1]{GarciaCastanoMelguizoParzanese2023}}]
\label{thm:primer_tma_separacion_C_y_K}
Let \(C,K\subset X\) be cones. The following assertions are equivalent:
\begin{enumerate}[label=\textup{(\roman*)}]
\item \((C,K)\) verifies the SSP. 
\item There exist \(0<\delta_1<\delta_2\) and \(f\in X^*\) such that
\((f,\alpha)\in C^{a\#}_+\) and
\(f(x)+\alpha\|x\|<0<f(y)+\alpha\|y\|\)
for all \(\alpha\in(\delta_1,\delta_2)\), all
\(x\in-\cl\co(C)\setminus\{0_X\}\), and all
\(y\in\bd(-K)\setminus\{0_X\}\).
\end{enumerate}
\end{theorem}

\subsection*{Co-radiant Sets}
\label{subsec:co-radiant-sets}

\begin{definition}
\label{def:co-radiant-set}
A set \(\cK\subset X\) is co-radiant if \(\lambda x\in\cK\) for every
\(x\in\cK\) and every \(\lambda\ge1\).
\end{definition}
Unless otherwise stated, co-radiant sets are nontrivial and do not contain the
origin: \(\varnothing\neq\cK\subset X\) and \(0_X\notin\cK\). A co-radiant set
\(\cK\) is solid if \(\intt \cK\neq\varnothing\). Every cone is co-radiant. If \(\eta>0\), set \(\cK(\eta):=\eta\cK\). For
\(A\subset X\), set \(\operatorname{sdw}(A):=\{\lambda a:a\in A,\ \lambda\ge1\}\).

\begin{definition}
\label{def:norm-base-coradiant-set}
Let \(\cK\subset X\) be co-radiant and let \(t>0\). The section
\(\cK_{[t]}:=\cK\cap tS_X\) is a norm-base of \(\cK\) if for every $x\in K$ there exist $y\in K_{[t]}$ and $\lambda>0$ such that $x=\lambda y$.
\end{definition}

\begin{proposition}
\label{prop:norm_base_tail_characterization}
A co-radiant set \(\cK\) admits a norm-base if and only if
\(\cK_{\ge t}=(\cone(\cK))_{\ge t}\) for some \(t>0\).
\end{proposition}

\begin{proof}
Assume that \(\cK\cap tS_X\) is a norm-base of \(\cK\). Since
\(\cK\subset\cone(\cK)\), it suffices to prove
\((\cone(\cK))_{\ge t}\subset\cK_{\ge t}\). Let
\(y\in(\cone(\cK))_{\ge t}\). Write \(y=\lambda k\), with \(\lambda>0\) and
\(k\in\cK\), and \(k=\mu k_t\), with \(\mu>0\) and
\(k_t\in\cK\cap tS_X\). Then \(y=\lambda\mu k_t\). Since
\(\|y\|\ge t=\|k_t\|\), one has \(\lambda\mu\ge1\), and therefore \(y\in\cK\).

Conversely, assume that \(\cK_{\ge t}=(\cone(\cK))_{\ge t}\). If \(k\in\cK\),
then \(k\neq0_X\) and \(t k/\|k\|\in(\cone(\cK))_{\ge t}=\cK_{\ge t}\). Hence
\(t k/\|k\|\in\cK\cap tS_X\), and
\(k=(\|k\|/t)(t k/\|k\|)\). Thus \(\cK\cap tS_X\) is a norm-base of \(\cK\).
\end{proof}

\subsection*{Basic Examples}
\label{subsec:basic-examples}

We start with a co-radiant set which admits a norm-base.

\begin{example}
\label{ex:K0-norm-base}
In \(X=\R^2\), set \(\cK_0:=\{z\in\Rp^2:\|z\|_2\ge1\}\). Then \(\cK_0\)
is co-radiant and admits a norm-base.
\end{example}

\begin{proof}
Co-radiance is immediate. Moreover, \(S_X\cap\Rp^2\) is a norm-base of
\(\cK_0\), since every \(z\in\cK_0\) satisfies
\(z=\|z\|_2\,z/\|z\|_2\) with \(z/\|z\|_2\in S_X\cap\Rp^2\).
\end{proof}

\noindent The next examples show that co-radiance alone does not ensure the existence of a norm-base.

\begin{example}
\label{ex:K1-no-norm-base}
In \(X=\R^2\), let \(\cK_1:=(1,1)+\Rp^2=\{(x,y)\in\R^2:x\ge1,\ y\ge1\}\).
Then \(\cK_1\) is co-radiant and has no norm-base.
\end{example}

\begin{proof}
Co-radiance is immediate. Let \(t>0\). For \(s\in(0,1)\), set
\(u_s:=(s,\sqrt{1-s^2})\in S_X\). The ray \(\Rp u_s\) meets \(\cK_1\), since
\(\rho u_s\in\cK_1\) for all sufficiently large \(\rho\). Choose
\(s_t\in(0,1)\) with \(t s_t<1\). Then \(t u_{s_t}\notin\cK_1\), although
\(\Rp u_{s_t}\cap\cK_1\neq\varnothing\). Hence \(\cK_1\cap tS_X\) is not a
norm-base. Since \(t\) is arbitrary, \(\cK_1\) has no norm-base.
\end{proof}

\begin{example}
\label{ex:K2-no-norm-base}
In \(X=\R^2\), let \(\cK_2:=\{(x,y):x>0,\ y\ge1/x\}\). Then \(\cK_2\)
is co-radiant and has no norm-base.
\end{example}

\begin{proof}
If \((x,y)\in\cK_2\) and \(\lambda\ge1\), then
\(\lambda y\ge\lambda/x\ge1/(\lambda x)\); hence \(\lambda(x,y)\in\cK_2\).

Let \(t>0\). For \(s\in(0,1)\), set \(u_s:=(s,\sqrt{1-s^2})\). Then
\(\Rp u_s\cap\cK_2\neq\varnothing\), while
\(t u_s\in\cK_2\) iff \(t^2s\sqrt{1-s^2}\ge1\). Choose \(s\in(0,1)\) with
\(t^2s\sqrt{1-s^2}<1\). Then \(t u_s\notin\cK_2\), so
\(\cK_2\cap tS_X\) is not a norm-base. Since \(t\) is arbitrary,
\(\cK_2\) has no norm-base.
\end{proof}

\noindent The same phenomenon also occurs for Bishop--Phelps type co-radiant sets.

\begin{example}
\label{ex:K3-no-norm-base}
Let \(X=\R^2\) be Euclidean, let \(f\in X^*\setminus\{0\}\),
\(0<\alpha<\|f\|_*\), and \(\lambda>0\). Set
\(\cK_3:=\{x\in X:f(x)-\alpha\|x\|\ge\lambda\}.\)
Then \(\cK_3\) is co-radiant and has no norm-base.
\end{example}

\begin{proof}
Co-radiance follows from the positive homogeneity of
\(x\mapsto f(x)-\alpha\|x\|\).

Let \(t>0\). Since \(0<\alpha<\|f\|_*\), the cone
\(C(f,\alpha):=\{x\in X:f(x)-\alpha\|x\|\ge0\}\)
is solid. Choose \(u\in S_X\cap\intt C(f,\alpha)\) such that
\(0<f(u)-\alpha<\lambda/t.\)
Then \(\Rp u\cap\cK_3\neq\varnothing\), but \(tu\notin\cK_3\). Hence
\(\cK_3\cap tS_X\) is not a norm-base. Since \(t>0\) is arbitrary, \(\cK_3\)
has no norm-base.
\end{proof}

\noindent The last example will be used later to show that regularity is a genuine additional property.

\begin{example}
\label{ex:K4-no-norm-base}
Let \(X=\R^2\) be Euclidean. For \(s\in [0,+\infty)\), set
\(u_s:=(1,s)/\sqrt{1+s^2}\), and set \(u_\infty:=(0,1)\). Define
\(\phi:[0,+\infty]\to[1,+\infty)\) by
\[
\phi(s):=
\begin{cases}
1, & s\notin\Q,\\[1mm]
q, & s=p/q\in\Q,\ \gcd(p,q)=1,\ q\ge1,
\end{cases}
\qquad
\phi(\infty):=1.
\]
Let \(\cK_4:=\{r u_s:s\in[0,+\infty],\ r\ge\phi(s)\}\). Then \(\cK_4\) is
co-radiant and has no norm-base.
\end{example}

\begin{proof}
Co-radiance follows directly from the definition. Let \(t>0\). Choose
\(q\in\N\) with \(q>t\), and set \(s_q:=1/q\). Then \(\phi(s_q)=q\). Hence
\(q u_{s_q}\in\cK_4\), whereas \(t u_{s_q}\notin\cK_4\). Thus
\(\cK_4\cap tS_X\) is not a norm-base. Since \(t\) is arbitrary, \(\cK_4\) has
no norm-base.
\end{proof}

\section{Inner Approximations of Cones and Stability of the SSP}
\label{sec:inner_approximations_cones}
In this section, we develop the conical framework underlying the notion of regularity. We begin by introducing the interior reductions of a cone and use them to define a new sequential approximation concept, namely, an inner approximating sequence. Such a sequence consists of nested closed cones contained in the original cone and is designed to eventually absorb all its interior reductions. The main result of this section, Theorem~\ref{teo:ssp_inner_approximating_sequence}, shows that the strict separation property (SSP) of a pair \((C,K)\) is inherited by \((C,K_n)\) for all sufficiently large indices. This stability property constitutes the key technical ingredient for the separation theorem established later, allowing the general case to be reduced to the norm-base setting through an appropriate approximating family of cones.

\subsection{Interior Reductions of Cones}
\label{subsec:interior_reductions}

Let \(K\) be a cone. We remove from \(K\) the rays generated by directions
close to \(\bd K\cap S_X\).

\begin{definition}
\label{def:interior_reduction_cone}
Let \(K\subset X\) be a cone. Set
\(\vee K:=\bd K\cap S_X.\)
For \(\varepsilon>0\), define
\[
\vee_\varepsilon K:=\vee K+\varepsilon B_X^\circ,
\qquad
\Sigma_\varepsilon(K):=\cone(\vee_\varepsilon K)\setminus\{0_X\}.
\]
The \(\varepsilon\)-interior reduction of \(K\) is
\(K^{-\varepsilon}:=K\setminus\Sigma_\varepsilon(K).\)
\begin{figure}[H]
    \centering
    \includegraphics[scale=0.112]{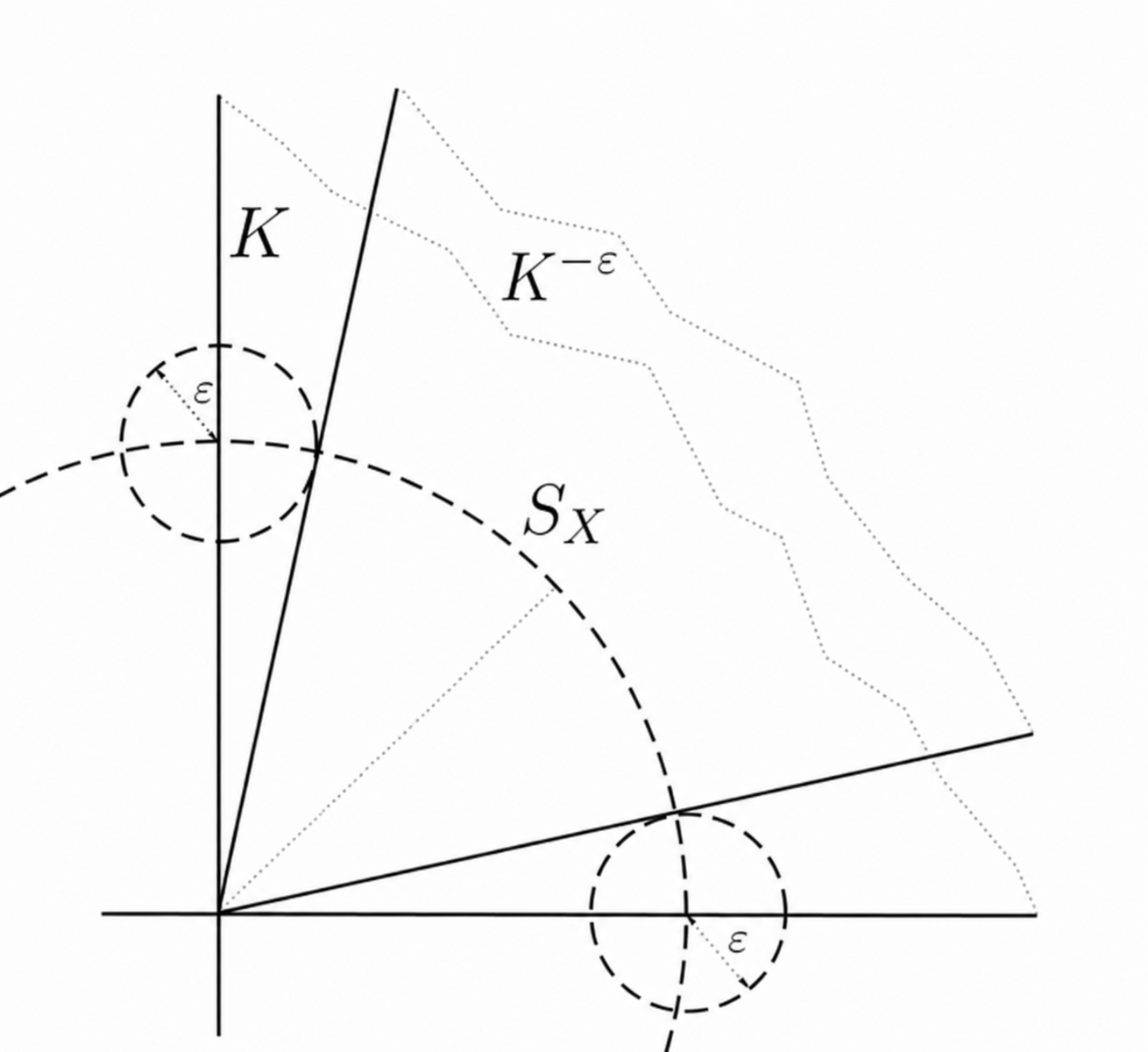}
    \caption{\(\varepsilon\)-interior reduction of \(K\).}
    \label{fig:K_epsilon}
\end{figure}
\end{definition}

Thus \(K^{-\varepsilon}\) is the part of \(K\) whose directions remain
uniformly separated from \(\bd K\cap S_X\).

\begin{lemma}
\label{lem:basic_properties_interior_reductions}
Let \(K\subset X\) be a cone. Then:
\begin{enumerate}[label=\textup{(\roman*)}]
\item \(K^{-\varepsilon}\subset K\) for every \(\varepsilon>0\).
\item If \(0<\varepsilon_1\le\varepsilon_2\), then
\(K^{-\varepsilon_2}\subset K^{-\varepsilon_1}\).
\item \(K^{-\varepsilon}\) is a closed cone.
\item If \(K\) is solid, then \(K^{-\varepsilon}\setminus\{0_X\}\neq\varnothing\)
for all sufficiently small \(\varepsilon>0\).
\end{enumerate}
\end{lemma}

\begin{proof}
Set \(\Sigma_\varepsilon(K):=\cone(\vee_\varepsilon K)\setminus\{0_X\}\).
Then \(K^{-\varepsilon}=K\setminus\Sigma_\varepsilon(K)\).

\textup{(i)} is immediate.

\textup{(ii)} If \(0<\varepsilon_1\le\varepsilon_2\), then
\[
\vee_{\varepsilon_1}K\subset\vee_{\varepsilon_2}K,
\qquad
\Sigma_{\varepsilon_1}(K)\subset\Sigma_{\varepsilon_2}(K).
\]
Hence \(K^{-\varepsilon_2}\subset K^{-\varepsilon_1}\).

\textup{(iii)} Let \(x\in K^{-\varepsilon}\) and \(\lambda\ge0\). Since \(K\) is
a cone, \(\lambda x\in K\). If \(\lambda=0\), then
\(\lambda x=0_X\notin\Sigma_\varepsilon(K)\). If \(\lambda>0\) and
\(\lambda x\in\Sigma_\varepsilon(K)\), then \(x\in\Sigma_\varepsilon(K)\), a
contradiction. Thus \(\lambda x\in K^{-\varepsilon}\), and
\(K^{-\varepsilon}\) is a cone.

Let \((x_n)\subset K^{-\varepsilon}\) and \(x_n\to x\). Since
\(\Sigma_\varepsilon(K)\) is open, \(x\notin\Sigma_\varepsilon(K)\). We prove
that \(x\in K\). This is clear if \(x=0_X\). Assume that \(x\neq0_X\) and
\(x\notin K\). Since \(x_n\in K\) and \(x_n\to x\), we have
\(x\in\cl K\setminus K\subset\bd K\). As \(K\) is a cone,
\(\frac{x}{\|x\|}\in\bd K\cap S_X=\vee K.\)
Hence
\(x\in\cone(\vee K)\setminus\{0_X\}\subset\Sigma_\varepsilon(K),\)
a contradiction. Therefore \(x\in K\), and so
\(x\in K\setminus\Sigma_\varepsilon(K)=K^{-\varepsilon}\).

\textup{(iv)} Choose \(u\in\intt K\cap S_X\). There exists \(\delta>0\) such that
\(u+\delta B_X \subset K\subset\intt K\). Hence \(\dist(u,\vee K)\ge\delta\). Let
\(0<\varepsilon<\min\{1,\delta/2\}.\)
If \(u\in\Sigma_\varepsilon(K)\), then
\(u=t(v+b)\)
for some \(t>0\), \(v\in\vee K\), and \(b\in B_X^\circ\) with
\(\|b\|<\varepsilon\). Since \(\|u\|=\|v\|=1\), one has
\(u=\frac{v+b}{\|v+b\|}.\)
Therefore
\(\|u-v\| \le |1-\|v+b\||+\|b\| <2\varepsilon<\delta,\)
which contradicts \(\dist(u,\vee K)\ge\delta\). Thus
\(u\in K^{-\varepsilon}\setminus\{0_X\}\).
\end{proof}

\subsection{Inner Approximating Sequences}
\label{subsec:sequential_inner_approximations}

We introduce a sequential notion of inner approximation by closed cones, tailored to the separation arguments below.
\begin{definition}
\label{def:inner_approximating_sequence}
Let \(K\subset X\) be a cone. A sequence \((K_n)\) of cones is an
\emph{inner approximating sequence} for \(K\) if:
\begin{enumerate}[label=\textup{(\roman*)}]
\item \(K_n\subset K_{n+1}\subset K\) for every \(n\in\N\).
\item For every \(\varepsilon>0\), there exists \(n_\varepsilon\in\N\) such that
\(K^{-\varepsilon}\subset K_n \qquad \forall n\ge n_\varepsilon.\)
\end{enumerate}
If, in addition, each \(K_n\) is closed, then \((K_n)\) is called a
\emph{closed inner approximating sequence}.
\end{definition}

\begin{proposition}
\label{prop:canonical_inner_approximating_sequence}
Let \(K\subset X\) be a solid cone. Then \((K^{-1/n})\) is a closed inner
approximating sequence for \(K\).
\end{proposition}

\begin{proof}
Set \(K_n:=K^{-1/n}\). By
Lemma~\ref{lem:basic_properties_interior_reductions}, each \(K_n\) is a
closed cone and \(K_n\subset K\). Since \(1/(n+1)<1/n\),
\(K_n=K^{-1/n}\subset K^{-1/(n+1)}=K_{n+1}.\)
Let \(\varepsilon>0\). Choose \(n_\varepsilon\in\N\) such that
\(1/n_\varepsilon\le\varepsilon\). If \(n\ge n_\varepsilon\), then
\(1/n\le\varepsilon\), and therefore
\(K^{-\varepsilon}\subset K^{-1/n}=K_n.\)
Thus \((K^{-1/n})\) is a closed inner approximating sequence for \(K\).
\end{proof}

\begin{remark}
By Lemma~\ref{lem:basic_properties_interior_reductions}\textup{(iv)},
the cones \(K^{-1/n}\) are nontrivial for all sufficiently large \(n\).
Whenever some initial terms are trivial, they will be tacitly discarded
and the resulting tail reindexed. This convention does not affect any of
the subsequent arguments.
\end{remark}

\begin{proposition}
\label{prop:inner_approximating_sequence_covers_interior}
Let \(K\subset X\) be a solid cone and let \((K_n)\) be an inner approximating
sequence for \(K\). Then
\(\intt K\subset\bigcup_{n\in\N}K_n\subset K.\)
\end{proposition}

\begin{proof}
The inclusion \(\bigcup_{n\in\N}K_n\subset K\) follows from the definition.

Let \(x\in\intt K\). If \(x=0_X\), then \(x\in K_n\) for every \(n\in\N\).
Assume that \(x\neq0_X\), and set \(u:=x/\|x\|\). Then
\(u\in\intt K\cap S_X\). Hence there exists \(\delta>0\) such that
\(u+\delta B_X \subset K\), and consequently \(\dist(u,\vee K)\ge\delta\).

Choose \(0<\varepsilon<\delta/2\). We prove that \(u\in K^{-\varepsilon}\).
If \(u\in\Sigma_\varepsilon(K)\), then \(u=t(v+b)\) for some \(t>0\),
\(v\in\vee K\), and \(b\in B_X^\circ\) with \(\|b\|<\varepsilon\). Since
\(\|u\|=\|v\|=1\),
\(u=\frac{v+b}{\|v+b\|},\)
and hence
\(\|u-v\| \le \left|1-\|v+b\|\right|+\|b\| <2\varepsilon<\delta,\)
a contradiction. Thus \(u\in K^{-\varepsilon}\). Since \(K^{-\varepsilon}\)
is a cone, \(x=\|x\|u\in K^{-\varepsilon}\).

By definition of inner approximating sequence, there exists \(n_x\in\N\)
such that \(K^{-\varepsilon}\subset K_{n_x}\). Hence \(x\in K_{n_x}\).
\end{proof}

\subsection{Stability of the SSP}
\label{subsec:ssp_stability_inner_approximations}

We prove that the SSP is stable under sufficiently accurate closed inner approximations of the second cone. \\

\noindent This lemma localizes the boundary of an interior reduction near \(\bd K\).

\begin{lemma}
\label{lem:boundary_interior_reduction_inclusion}
Let \(K\subset X\) be a cone and let \(0<\varepsilon<1\). Then
\[
\bd(K^{-\varepsilon})\setminus\{0_X\}
\subset
\cone\bigl(\vee K+3\varepsilon B_X^\circ\bigr).
\]
\end{lemma}

\begin{proof}
Let \(y\in\bd(K^{-\varepsilon})\setminus\{0_X\}\). Since
\(K^{-\varepsilon}=K\setminus\Sigma_\varepsilon(K)\),
\(\bd(K^{-\varepsilon})\subset \bd K\cup\bd\Sigma_\varepsilon(K).\)
If \(y\in\bd K\), then \(y=\|y\|v\) for some \(v\in\vee K\), and the conclusion
follows.

Assume that \(y\in\bd\Sigma_\varepsilon(K)\). Then
\(y\in\cl\Sigma_\varepsilon(K)\). Since \(y\neq0_X\), the direction
\(y/\|y\|\) belongs to the closure of the normalized directions generated by
\(\vee K+\varepsilon B_X^\circ\). Hence
\(\dist\left(\frac{y}{\|y\|},\vee K\right)\le2\varepsilon,\)
and therefore
\(\frac{y}{\|y\|}\in\vee K+3\varepsilon B_X^\circ.\)
Thus \(y\in\cone(\vee K+3\varepsilon B_X^\circ)\).
\end{proof}

\noindent The next proposition establishes the SSP stability for small interior reductions.

\begin{proposition}
\label{prop:reduced_cone_boundary_separation}
Let \(C,K\subset X\) be cones such that \((C,K)\) verifies the SSP. Then there
exists \(\varepsilon_0>0\) such that, for every
\(0<\varepsilon<\varepsilon_0\), the pair \((C,K^{-\varepsilon})\) verifies the
SSP whenever \(K^{-\varepsilon}\setminus\{0_X\}\neq\varnothing\).
\end{proposition}

\begin{proof}
Since \((C,K)\) verifies the SSP, \cite[Theorem 3.1]{GarciaCastanoMelguizoParzanese2023}, applied
to \((-C,-K)\), yields \(0<\delta_1<\delta_2\) and \(f\in X^*\) such that
\(f(x)+\alpha\|x\|<0<f(y)+\alpha\|y\|\)
for every \(\alpha\in(\delta_1,\delta_2)\), every
\(x\in\clco(C)\setminus\{0_X\}\), and every
\(y\in\bd K\setminus\{0_X\}\).

Choose
\(0<\delta_1<\beta<\alpha_1<\alpha_2<\delta_2\)
and set
\(\varepsilon_0:= \frac{\alpha_1-\beta}{3(\|f\|+\alpha_1)}.\)
Let \(0<\varepsilon<\varepsilon_0\), \(\alpha\in(\alpha_1,\alpha_2)\), and
\(y\in\cone(\vee K+3\varepsilon B_X^\circ)\setminus\{0_X\}.\)
Then \(y=\lambda(s+3\varepsilon b)\), with
\(\lambda>0\), \(s\in\vee K\), and \(b\in B_X^\circ\). Hence
\[
\begin{aligned}
f(y)+\alpha\|y\|
&> f(y)+\alpha_1\|y\| \\
&\ge f(\lambda s)+\alpha_1\|\lambda s\|
      -3\lambda\varepsilon(\|f\|+\alpha_1) \\
&> f(\lambda s)+\beta\|\lambda s\| \\
&>0.
\end{aligned}
\]
Here we used \(3\varepsilon(\|f\|+\alpha_1)<\alpha_1-\beta\) and
\(\|\lambda s\|=\lambda\).

by construction,
\[
f(x)+\alpha\|x\|<0
\qquad
\forall x\in\clco(C)\setminus\{0_X\},\quad
\forall\alpha\in(\alpha_1,\alpha_2).
\]
By Lemma~\ref{lem:boundary_interior_reduction_inclusion},
\[
\bd(K^{-\varepsilon})\setminus\{0_X\}
\subset
\cone(\vee K+3\varepsilon B_X^\circ).
\]
Therefore
\(f(x)+\alpha\|x\|<0<f(y)+\alpha\|y\|\)
for every \(x\in\clco(C)\setminus\{0_X\}\), every
\(y\in\bd(K^{-\varepsilon})\setminus\{0_X\}\), and every
\(\alpha\in(\alpha_1,\alpha_2)\).

If \(K^{-\varepsilon}\setminus\{0_X\}\neq\varnothing\), the characterization of
the SSP, applied to \((-C,-K^{-\varepsilon})\), gives that
\((C,K^{-\varepsilon})\) verifies the SSP.
\end{proof}

\noindent For solid cones, nontriviality is automatic for small reductions.

\begin{corollary}
\label{cor:reduced_cone_SSP_solid}
Let \(C,K\subset X\) be cones such that \((C,K)\) verifies the SSP and
\(\intt K\neq\varnothing\). Then there exists \(\varepsilon_0>0\) such that,
for every \(0<\varepsilon<\varepsilon_0\),
\[
K^{-\varepsilon}\setminus\{0_X\}\neq\varnothing
\quad\text{and}\quad
(C,K^{-\varepsilon}) \text{ verifies the SSP}.
\]
\end{corollary}

\begin{proof}
By Lemma~\ref{lem:basic_properties_interior_reductions}, there exists
\(\mu>0\) such that
\(K^{-\varepsilon}\setminus\{0_X\}\neq\varnothing\) for every
\(0<\varepsilon<\mu\). By
Proposition~\ref{prop:reduced_cone_boundary_separation}, there exists
\(\eta>0\) such that \((C,K^{-\varepsilon})\) verifies the SSP whenever
\(0<\varepsilon<\eta\) and
\(K^{-\varepsilon}\setminus\{0_X\}\neq\varnothing\). Take
\(\varepsilon_0:=\min\{\mu,\eta\}.\)
\end{proof}

\noindent The same conclusion holds under the usual nontrivial intersection condition.

\begin{corollary}
\label{cor:reduced_cone_SSP_intersection}
Let \(C,K\subset X\) be cones such that \((C,K)\) verifies the SSP and
\(C\cap K\neq\{0_X\}\). Then there exists \(\varepsilon_0>0\) such that, for
every \(0<\varepsilon<\varepsilon_0\),
\[
K^{-\varepsilon}\setminus\{0_X\}\neq\varnothing
\quad\text{and}\quad
(C,K^{-\varepsilon}) \text{ verifies the SSP}.
\]
\end{corollary}

\begin{proof}
By \cite[Theorem 3.3]{GarciaCastanoMelguizoParzanese2023}, the SSP and
\(C\cap K\neq\{0_X\}\) imply
\(C\setminus\{0_X\}\subset\intt K.\)
In particular, \(\intt K\neq\varnothing\). The result follows from
Corollary~\ref{cor:reduced_cone_SSP_solid}.
\end{proof}

\noindent We shall also use the following monotonicity property of the SSP.

\begin{corollary}[\cite{HenigDensity2024}, Corollary 3.14]
\label{coro:inclusiones_y_SSP}
Let \(X\) be a normed space and let \(C,K_1,K_2\subset X\) be cones such that
\(C\cap K_1\neq\{0_X\}\). If \((C,K_1)\) verifies the SSP and \(K_1\subset K_2\),
then \((C,K_2)\) verifies the SSP.
\end{corollary}

\noindent Stability and monotonicity yield SSP for intermediate closed cones.

\begin{proposition}
\label{prop:intermediate_cones_ssp}
Let \(C,K\subset X\) be cones such that \((C,K)\) verifies the SSP and
\(C\cap K\neq\{0_X\}\). Then there exists \(\varepsilon>0\) such that, for
every closed cone \(L\) satisfying
\(K^{-\varepsilon}\subset L\subset K,\)
the pair \((C,L)\) verifies the SSP.
\end{proposition}

\begin{proof}
By \cite[Theorem 3.3]{GarciaCastanoMelguizoParzanese2023},
\(C\setminus\{0_X\}\subset\intt K.\)
Choose \(z\in C\cap K\), \(z\neq0_X\). Then \(z\in\intt K\).

Arguing as in Lemma~\ref{lem:basic_properties_interior_reductions}, there
exists \(\eta>0\) such that
\(z\in K^{-\varepsilon'} \quad \forall\,0<\varepsilon'<\eta.\)
By Proposition~\ref{prop:reduced_cone_boundary_separation}, decreasing
\(\eta\) if necessary, choose \(0<\varepsilon<\eta\) such that
\((C,K^{-\varepsilon})\) verifies the SSP. Then
\(C\cap K^{-\varepsilon}\neq\{0_X\}.\)
Let \(L\) be a closed cone such that
\(K^{-\varepsilon}\subset L\subset K\). Since
\((C,K^{-\varepsilon})\) verifies the SSP,
\(C\cap K^{-\varepsilon}\neq\{0_X\}\), and
\(K^{-\varepsilon}\subset L\), Corollary~\ref{coro:inclusiones_y_SSP} gives
that \((C,L)\) verifies the SSP.
\end{proof}

\noindent We now pass to closed inner approximating sequences.

\begin{theorem}
\label{teo:ssp_inner_approximating_sequence}
Let \(C,K\subset X\) be cones such that \((C,K)\) verifies the SSP, and let
\((K_n)\) be a closed inner approximating sequence for \(K\). Then there
exists \(n_0\in\N\) such that \((C,K_n)\) verifies the SSP for every
\(n\ge n_0\).
\end{theorem}

\begin{proof}
Assume first that \(C\cap K\neq\{0_X\}\). By
Proposition~\ref{prop:intermediate_cones_ssp}, there exists
\(\varepsilon>0\) such that \((C,L)\) verifies the SSP for every closed cone
\(L\) satisfying
\(K^{-\varepsilon}\subset L\subset K.\)
Since \((K_n)\) is an inner approximating sequence for \(K\), there exists
\(n_0\in\N\) such that
\(K^{-\varepsilon}\subset K_n \qquad \forall n\ge n_0.\)
For \(n\ge n_0\), the cone \(K_n\) is closed and satisfies
\(K^{-\varepsilon}\subset K_n\subset K.\)
Hence \((C,K_n)\) verifies the SSP.

Assume now that \(C\cap K=\{0_X\}\). For \(n\in\N\), set
\(K':=(X\setminus K)\cup\{0_X\}, \quad K_n':=(X\setminus K_n)\cup\{0_X\}.\)
Since \(K_n\subset K\), one has \(K'\subset K_n'\). Also, the SSP for
\((C,K)\) is equivalent to the SSP for \((C,K')\). Choose
\(c\in C\setminus\{0_X\}\). Since \(C\cap K=\{0_X\}\), one has
\(c\notin K\), and hence \(c\in K'\). Thus
\(C\cap K'\neq\{0_X\}.\)
By Corollary~\ref{coro:inclusiones_y_SSP}, applied to \(C\), \(K'\), and
\(K_n'\), the pair \((C,K_n')\) verifies the SSP for every \(n\in\N\). By the same
equivalence, \((C,K_n)\) verifies the SSP for every \(n\in\N\). The conclusion
holds with any \(n_0\in\N\).
\end{proof}

\noindent We shall use the preceding construction only through
Theorem~\ref{teo:ssp_inner_approximating_sequence}.

\section{Regular Co-radiant Sets and Characterization via Inner
Approximations}
\label{sec:regular_coradiant_sets}
We introduce a regularity condition for co-radiant sets. It replaces the
existence of a global norm-base by the existence of norm-bases after removing
directions close to the boundary of the generated cone.

Let \(\mathcal K\subset X\) be a co-radiant set. We next show how the reductions of the cone generated by \(\mathcal K\), introduced in Definition~\ref{def:interior_reduction_cone}, naturally induce corresponding reductions of the original co-radiant set \(\mathcal K\). These induced reductions will subsequently serve as the basis for introducing the notion of regular co-radiant sets.

\begin{definition}
\label{def:interior_reduction_coradiant_set}
Let \(\cK\subset X\) be co-radiant and let \(\varepsilon>0\). The
\(\varepsilon\)-interior reduction of \(\cK\) is
\(\cK^{-\varepsilon} := \cK\cap(\cone(\cK))^{-\varepsilon}.\)

\begin{figure}[H]
    \centering
    \includegraphics[scale=0.12]{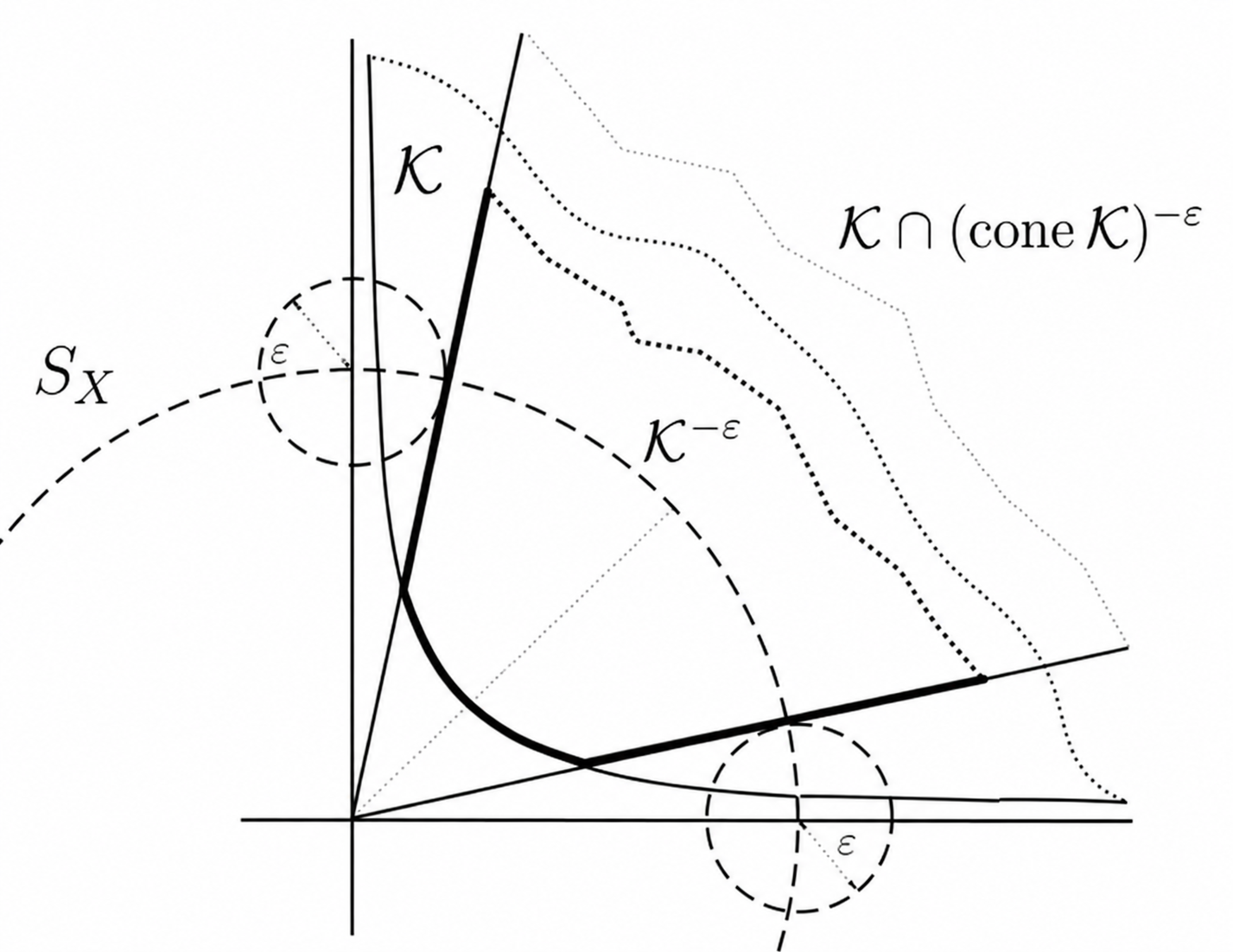}
    \caption{\(\varepsilon\)-interior reduction of \(\cK\).}
    \label{fig:coradiant_K_epsilon}
\end{figure}
\end{definition}

Since \((\cone(\cK))^{-\varepsilon}\) is a cone, the set
\(\cK^{-\varepsilon}\) is again co-radiant.

\begin{definition}
\label{def:regular_coradiant_set}
Let \(\cK\subset X\) be co-radiant. We say that \(\cK\) is regular if, for
every \(\varepsilon>0\) such that \(\cK^{-\varepsilon}\neq\varnothing\), the
set \(\cK^{-\varepsilon}\) admits a norm-base.
\end{definition}

\noindent Thus regularity does not require \(\cK\) itself to admit a norm-base.
We give a sequential form of regularity.

\begin{proposition}
\label{prop:regularity_inner_approx_characterization}
Let \(\cK\) be co-radiant and suppose that \(K:=\cone(\cK)\) is solid. Then
the following assertions are equivalent:
\begin{enumerate}[label=\textup{(\roman*)}]
\item \(\cK\) is regular.
\item There exist an inner approximating sequence \((K_n)\) for \(K\) and
\(n_0\in\N\) such that \(\cK\cap K_n\) admits a norm-base for every
\(n\ge n_0\).
\end{enumerate}
\end{proposition}

\begin{proof}
\textup{(i)}\(\Rightarrow\)\textup{(ii)}. Assume that \(\cK\) is regular.
Set \(K_n:=K^{-1/n}\). By
Proposition~\ref{prop:canonical_inner_approximating_sequence}, \((K_n)\) is
a closed inner approximating sequence for \(K\).

We first prove that \(\cK\cap\intt K\neq\varnothing\). Otherwise
\(\cK\subset K\setminus\intt K\subset\bd K\). Since \(K\) is a cone, every
positive multiple of a nonzero point of \(\bd K\) again belongs to \(\bd K\).
Hence \(\cone(\cK)\subset\bd K\cup\{0_X\}\), contradicting the solidity of
\(K=\cone(\cK)\).

Choose \(x\in\cK\cap\intt K\). By
Proposition~\ref{prop:inner_approximating_sequence_covers_interior}, there
exists \(n_0\in\N\) such that \(x\in K_n\) for every \(n\ge n_0\). Hence
\(\cK\cap K_n\neq\varnothing\) for every \(n\ge n_0\). Moreover,
\(\cK\cap K_n = \cK\cap K^{-1/n} = \cK^{-1/n} \quad(n\ge n_0).\)
By regularity, \(\cK^{-1/n}\) admits a norm-base. Thus \(\cK\cap K_n\)
admits a norm-base for every \(n\ge n_0\).

\textup{(ii)}\(\Rightarrow\)\textup{(i)}. Assume \textup{(ii)}. Let
\(\varepsilon>0\) be such that \(\cK^{-\varepsilon}\neq\varnothing\). Since
\((K_n)\) is an inner approximating sequence for \(K\), there exists
\(n_\varepsilon\in\N\) such that \(K^{-\varepsilon}\subset K_n\) for every
\(n\ge n_\varepsilon\). Choose \(n\ge\max\{n_0,n_\varepsilon\}\). Then
\(\cK^{-\varepsilon} = \cK\cap K^{-\varepsilon} \subset \cK\cap K_n.\)
Let \(t_n>0\) be such that \((\cK\cap K_n)\cap t_nS_X\) is a norm-base of
\(\cK\cap K_n\). We prove that \(\cK^{-\varepsilon}\cap t_nS_X\) is a
norm-base of \(\cK^{-\varepsilon}\).

Let \(y\in\cK^{-\varepsilon}\). Since \(y\in\cK\cap K_n\), there are
\(\lambda>0\) and \(u\in(\cK\cap K_n)\cap t_nS_X\) such that \(y=\lambda u\).
Since \(y\in K^{-\varepsilon}\) and \(K^{-\varepsilon}\) is a cone, one has
\(u=\lambda^{-1}y\in K^{-\varepsilon}\). Also \(u\in\cK\). Hence
\(u\in\cK^{-\varepsilon}\cap t_nS_X\). Therefore
\(\cK^{-\varepsilon}\cap t_nS_X\) is a norm-base of
\(\cK^{-\varepsilon}\). Thus \(\cK\) is regular.
\end{proof}

Regularity is therefore an eventual norm-base property along inner
approximations of \(\cone(\cK)\).

The next elementary fact will be used to identify the cone generated by an
interior part of a co-radiant set.

\begin{proposition}
\label{prop:cone_of_intersection_with_subcone}
Let \(\cK\subset X\) be co-radiant, and let \(L\subset\cone(\cK)\) be a
nontrivial cone. Then
\(\cone(\cK\cap L)=L.\)
\end{proposition}

\begin{proof}
Since \(\cK\cap L\subset L\) and \(L\) is a cone, one has
\(\cone(\cK\cap L)\subset L.\)
We prove the reverse inclusion. Let \(x\in L\setminus\{0_X\}\). Since
\(L\subset\cone(\cK)\), there exist \(\lambda>0\) and \(y\in\cK\) such that
\(x=\lambda y.\)
If \(\lambda\ge1\), then \(x=\lambda y\in\cK\), because \(\cK\) is
co-radiant. Hence \(x\in\cK\cap L\).

If \(0<\lambda<1\), then
\(y=\lambda^{-1}x.\)
Since \(L\) is a cone, \(y\in L\). Hence \(y\in\cK\cap L\), and therefore
\(x=\lambda y\in\cone(\cK\cap L)\).

Thus \(L\setminus\{0_X\}\subset\cone(\cK\cap L)\). Since \(L\) is nontrivial,
\(\cK\cap L\neq\varnothing\), and hence \(0_X\in\cone(\cK\cap L)\). Therefore
\(L\subset\cone(\cK\cap L)\).
\end{proof}

\section{Radial Profile Functions for Co-radiant Sets}
\label{sec:radial_profiles_coradiant_sets}

The radial profile function is naturally associated with every co-radiant set. As will be shown in Subsection~\ref{subsec:regularity_radial_profile}, this function provides a complete characterization of the existence of norm bases for co-radiant sets. Moreover, together with the inner approximating sequences introduced in the previous section, it plays a fundamental role in the characterization of regularity under a mild solidness assumption on the associated cone. In Subsection~\ref{subsec:radial_scalarization_approximate_efficiency}, we further show that the radial profile function gives rise to a characteristic-type scalar function associated with a co-radiant set, leading to natural scalarization results for approximate efficiency problems.

\begin{definition}
\label{def:radial-profile}
Let \(\cK\) be co-radiant. Its radial profile is defined by
\[
\rho_{\cK}:S_X\to[0,+\infty],
\qquad
\rho_{\cK}(v):=\inf\{\lambda>0:\lambda v\in\cK\},
\]
with the convention $\inf \varnothing=+\infty$. Set
\[
\dom(\rho_{\cK}):=\{v\in S_X:\rho_{\cK}(v)<+\infty\},
\quad
d_{\cK}:=\inf_{v\in\dom(\rho_{\cK})}\rho_{\cK}(v),
\quad
\mathcal I_{\cK}:=\sup_{v\in\dom(\rho_{\cK})}\rho_{\cK}(v).
\]
\end{definition}
The number \(\rho_{\cK}(v)\) is the radial threshold of \(\cK\) along the ray
\(\Rp v\). The threshold need not be attained.

\begin{lemma}
\label{lem:basic_radial_profile}
Let \(\cK\subset X\) be co-radiant. Then:
\begin{enumerate}[label=\textup{(\roman*)}]
\item \(\dom(\rho_{\cK})=\cone(\cK)\cap S_X\).
\item \(d_{\cK}=\inf\{\|k\|:k\in\cK\}= \dist(0_X,\cK)\).
\item
\(
\mathcal I_{\cK}
=
\inf\{t>0:\cK\cap tS_X\text{ is a norm-base of }\cK\}.
\)
\end{enumerate}
\end{lemma}

\begin{proof}
\textup{(i)} follows from the definitions.

\textup{(ii)} Let \(k\in\cK\). Since \(0_X\notin\cK\), one has \(k\neq0_X\).
Set \(v:=k/\|k\|\). Then \(v\in\dom(\rho_{\cK})\) and
\(
\rho_{\cK}(v)\le\|k\|.
\)
Hence
\(d_{\cK}\le\inf\{\|k\|:k\in\cK\}.\)
Conversely, if \(v\in\dom(\rho_{\cK})\) and \(\lambda v\in\cK\), then
\(\inf\{\|k\|:k\in\cK\}\le\lambda.\)
Taking infima over such \(\lambda\) and then over \(v\in\dom(\rho_{\cK})\)
gives the reverse inequality. The identity with \(d_\cK(0_X)\) is
immediate.

\textup{(iii)} Set
\(T:=\{t>0:\cK\cap tS_X\text{ is a norm-base of }\cK\}.\)
If \(t\in T\) and \(v\in\dom(\rho_{\cK})\), then \(tv\in\cK\). Hence
\(\rho_{\cK}(v)\le t\), and therefore
\(\mathcal I_{\cK}\le\inf T.\)
Conversely, assume first that \(\mathcal I_{\cK}<+\infty\), and let
\(t>\mathcal I_{\cK}\). For every \(v\in\dom(\rho_{\cK})\), there exists
\(\lambda<t\) such that \(\lambda v\in\cK\). By co-radiance, \(tv\in\cK\).
Thus every \(k\in\cK\), with \(v=k/\|k\|\), is a positive multiple of
\(tv\in\cK\cap tS_X\). Hence \(t\in T\), and so
\(\inf T\le\mathcal I_{\cK}.\)
If \(\mathcal I_{\cK}=+\infty\), the inequality
\(\mathcal I_{\cK}\le\inf T\) forces \(T=\varnothing\). Hence the equality
also holds in this case.
\end{proof}

\begin{remark}
\label{rem:norm_base_bounded_radial_profile}
Let \(\cK\subset X\) be co-radiant. Then \(\cK\) admits a norm-base if and only if
\(\rho_{\cK}\) is bounded above on \(\dom(\rho_{\cK})\). Equivalently,
\(\mathcal I_{\cK}<+\infty.\)
\end{remark}

\noindent We first compute the profile of the translated quadrant.

\begin{example}
\label{ex:K1-radial-profile}
For the set \(\cK_1\) of Example~\ref{ex:K1-no-norm-base},
\[
\rho_{\cK_1}(a,b)=
\begin{cases}
\max\{1/a,1/b\}, & (a,b)\in S_X,\ a>0,\ b>0,\\
+\infty, & \text{otherwise}.
\end{cases}
\]
Consequently,
\[
\dom(\rho_{\cK_1})=\{(a,b)\in S_X:a>0,\ b>0\},
\qquad
d_{\cK_1}=\sqrt2,
\qquad
\mathcal I_{\cK_1}=+\infty.
\]
\end{example}

\begin{proof}
Let \(v=(a,b)\in S_X\). Then \(\lambda v\in\cK_1\) if and only if
\(\lambda a\ge1 \quad\text{and}\quad \lambda b\ge1.\)
This gives the formula. The minimum of \(\max\{1/a,1/b\}\) on the open
first quadrant of \(S_X\) is attained at \(a=b=1/\sqrt2\). The supremum is
\(+\infty\) as \(a\downarrow0\) or \(b\downarrow0\).
\end{proof}

\noindent The next example has the same domain but a different radial growth.

\begin{example}
\label{ex:K2-radial-profile}
For \(\cK_2\) of Example~\ref{ex:K2-no-norm-base},
\[
\rho_{\cK_2}(a,b)=
\begin{cases}
1/\sqrt{ab}, & (a,b)\in S_X,\ a>0,\ b>0,\\
+\infty, & \text{otherwise}.
\end{cases}
\]
Hence \(\dom(\rho_{\cK_2})=\{(a,b)\in S_X:a>0,\ b>0\}\),
\(d_{\cK_2}=\sqrt2\), and \(\mathcal I_{\cK_2}=+\infty\).
\end{example}

\begin{proof}
For \(v=(a,b)\in S_X\), one has \(\lambda v\in\cK_2\) iff
\(a>0\), \(b>0\), and \(\lambda^2ab\ge1\). This gives the formula. The
minimum follows from \(\max\{ab:a^2+b^2=1,\ a,b>0\}=1/2\), and the
supremum is \(+\infty\) as \(a\downarrow0\) or \(b\downarrow0\).
\end{proof}

\noindent Bishop--Phelps type co-radiant sets have an equally explicit profile.

\begin{example}
\label{ex:K3-radial-profile}
For \(\cK_3\) of Example~\ref{ex:K3-no-norm-base},
\[
\rho_{\cK_3}(v)=
\begin{cases}
\dfrac{\lambda}{\langle f,v\rangle-\alpha},
& v\in S_X,\ \langle f,v\rangle>\alpha,\\[2mm]
+\infty,
& v\in S_X,\ \langle f,v\rangle\le\alpha.
\end{cases}
\]
Hence \(\dom(\rho_{\cK_3})=\{v\in S_X:\langle f,v\rangle>\alpha\}\),
\(d_{\cK_3}=\lambda/(\|f\|_2-\alpha)\), and
\(\mathcal I_{\cK_3}=+\infty\).
\end{example}

\begin{proof}
For \(v\in S_X\), one has \(\mu v\in\cK_3\) iff
\(\mu(\langle f,v\rangle-\alpha)\ge\lambda\). This gives the formula and
the domain. The minimum is obtained at \(v=f/\|f\|_2\). The supremum is
\(+\infty\) as \(v\) approaches
\(\{v\in S_X:\langle f,v\rangle=\alpha\}\) from the domain.
\end{proof}

\noindent The final profile records the unbounded oscillation of the pathological example.

\begin{example}
\label{ex:K4-radial-profile}
For the set \(\cK_4\) of Example~\ref{ex:K4-no-norm-base}, one has
\(\rho_{\cK_4}(u_s)=\phi(s) \qquad \forall s\in[0,+\infty].\)
Hence
\[
\dom(\rho_{\cK_4})=S_X\cap\Rp^2,
\qquad
d_{\cK_4}=1,
\qquad
\mathcal I_{\cK_4}=+\infty.
\]
\end{example}

\begin{proof}
By definition, \(ru_s\in\cK_4\) if and only if \(r\ge\phi(s)\). Thus
\(\rho_{\cK_4}(u_s)=\phi(s).\)
The directions \(u_s\), \(s\in[0,+\infty]\), are exactly
\(S_X\cap\Rp^2\). Since \(\phi\ge1\) and \(\phi(s)=1\) for irrational \(s\),
one has \(d_{\cK_4}=1\). Finally, \(\mathcal I_{\cK_4}=+\infty\), because
\(\phi\) is unbounded on every nondegenerate interval contained in
\(\R_{++}\). For instance,
\(s_q:=\frac{q+1}{q}\in[1,2] \quad\text{satisfies}\quad \phi(s_q)=q.\)
\end{proof}

\subsection{Radial Profile Characterizations of Norm-bases and Regularity}
\label{subsec:regularity_radial_profile}
This subsection is devoted to characterizations of norm bases and regularity in co-radiant sets in terms of the radial profile function. We further return to the examples of co-radiant sets introduced earlier and apply the obtained criteria to study their regularity. In particular, we prove that three of them are regular, whereas one is not.

\begin{lemma}
\label{lem:norm_base_radial_profile}
Let \(\cK\subset X\) be co-radiant set. Then \(\cK\) admits a norm-base if and only if
\(\rho_{\cK}\) is bounded above on \(\dom(\rho_{\cK})\). In this case,
\(\cK\) is regular. 
\end{lemma}

\begin{proof}
By Lemma~\ref{lem:basic_radial_profile},
\(\mathcal I_{\cK} = \inf\{t>0:\cK\cap tS_X\text{ is a norm-base of }\cK\}.\)
Thus \(\cK\) admits a norm-base if and only if
\(\mathcal I_{\cK}<+\infty,\)
which is equivalent to boundedness of \(\rho_{\cK}\) on
\(\dom(\rho_{\cK})\).

Assume that \(\cK\) admits a norm-base. Let \(t>0\) be such that
\(\cK\cap tS_X\) is a norm-base of \(\cK\). Let \(\varepsilon>0\) satisfy
\(\cK^{-\varepsilon}\neq\varnothing.\)
We prove that \(\cK^{-\varepsilon}\cap tS_X\) is a norm-base of
\(\cK^{-\varepsilon}\).

Let \(y\in\cK^{-\varepsilon}\). There exist \(\lambda>0\) and
\(u\in\cK\cap tS_X\) such that
\(y=\lambda u.\)
Since \(y\in(\cone(\cK))^{-\varepsilon}\) and
\((\cone(\cK))^{-\varepsilon}\) is a cone,
\(u=\lambda^{-1}y\in(\cone(\cK))^{-\varepsilon}.\)
Thus
\(u\in\cK\cap(\cone(\cK))^{-\varepsilon}\cap tS_X = \cK^{-\varepsilon}\cap tS_X.\)
Hence \(\cK^{-\varepsilon}\cap tS_X\) is a norm-base of
\(\cK^{-\varepsilon}\). Therefore \(\cK\) is regular.
\end{proof}

\noindent A global Lipschitz bound is a simple sufficient condition.

\begin{corollary}
\label{cor:lipschitz_radial_profile_norm_base}
Let \(\cK\subset X\) be co-radiant set. If \(\rho_{\cK}\) is Lipschitz continuous on
\(\dom(\rho_{\cK})\), then \(\cK\) admits a norm-base. In particular,
\(\cK\) is regular.
\end{corollary}
\begin{proof}
Fix \(v_0\in\dom(\rho_{\cK})\). For every \(u\in\dom(\rho_{\cK})\),
\[
\rho_{\cK}(u)
\le
\rho_{\cK}(v_0)+\Lip(\rho_{\cK})\|u-v_0\|
\le
\rho_{\cK}(v_0)+2\Lip(\rho_{\cK}),
\]
because \(\dom(\rho_{\cK})\subset S_X\). Hence \(\rho_{\cK}\) is bounded
above on \(\dom(\rho_{\cK})\). The result follows from
Lemma~\ref{lem:norm_base_radial_profile}.
\end{proof}

\noindent Regularity only requires boundedness on inner spherical sections.

\begin{proposition}
\label{prop:regularity_bounded_radial_profile_characterization}
Let \(\cK\subset X\) be a co-radiant set, and let \(K:=\cone(\cK)\). Suppose that \(K\) is solid.
Then the following assertions are equivalent:
\begin{enumerate}[label=\textup{(\roman*)}]
\item \(\cK\) is regular.
\item There exist an inner approximating sequence of cones \((K_n)\) for \(K\)
and \(n_0\in\N\) such that, for every \(n\ge n_0\),
\(\rho_{\cK}\) is bounded above on \(S_X\cap K_n\).
\end{enumerate}
\end{proposition}

\begin{proof}
Let \((K_n)\) be an inner approximating sequence for \(K\), and set
\(A_n:=\cK\cap K_n \qquad(n\in\N).\)
Each \(A_n\) is co-radiant.

Assume that \(A_n\neq\varnothing\). We claim that
\(\dom(\rho_{A_n})=S_X\cap K_n\)
and
\(\rho_{A_n}(v)=\rho_{\cK}(v) \qquad \forall v\in S_X\cap K_n.\)
Let \(v\in S_X\cap K_n\). Since \(K_n\subset K=\cone(\cK)\), there exists
\(\lambda>0\) such that \(\lambda v\in\cK\). Since \(K_n\) is a cone,
\(\lambda v\in K_n\). Hence \(\lambda v\in A_n\), and so
\(v\in\dom(\rho_{A_n})\). Conversely, if \(v\in\dom(\rho_{A_n})\), then
\(\lambda v\in A_n\subset K_n\) for some \(\lambda>0\). Since \(K_n\) is a
cone, \(v\in K_n\). This proves the identity of domains. Also, for
\(v\in S_X\cap K_n\) and \(\lambda>0\),
\(\lambda v\in A_n \quad\Longleftrightarrow\quad \lambda v\in\cK.\)
Hence the two profiles coincide on \(S_X\cap K_n\).

Hence whenever \(A_n\neq\varnothing\), \(\rho_{\cK}\) is bounded
above on \(S_X\cap K_n\) if and only if \(\rho_{A_n}\) is bounded above on
\(\dom(\rho_{A_n})\). By Lemma~\ref{lem:norm_base_radial_profile}, this is
equivalent to saying that \(A_n\) admits a norm-base.

Assume first that \(\cK\) is regular. By
Proposition~\ref{prop:regularity_inner_approx_characterization}, there exist
an inner approximating sequence \((K_n)\) for \(K\) and \(n_0\in\N\) such
that \(A_n=\cK\cap K_n\) admits a norm-base for every \(n\ge n_0\). By the
previous paragraph, \(\rho_{\cK}\) is bounded above on \(S_X\cap K_n\) for
every \(n\ge n_0\). Thus \textup{(ii)} holds.

Conversely, assume \textup{(ii)}. We prove that \(A_n\neq\varnothing\) for
all sufficiently large \(n\). Since \(K=\cone(\cK)\) is solid,
\(\cK\cap\intt K\neq\varnothing.\)
Indeed, if \(\cK\cap\intt K=\varnothing\), then
\(\cK\subset K\setminus\intt K\subset\bd K.\)
Since \(K\) is a cone, every positive multiple of a nonzero boundary point
of \(K\) belongs to \(\bd K\). Hence
\(\cone(\cK)\subset\bd K\cup\{0_X\},\)
contrary to the solidity of \(K=\cone(\cK)\).

Choose \(x\in\cK\cap\intt K\). Since \((K_n)\) is an inner approximating
sequence for \(K\), Proposition~\ref{prop:inner_approximating_sequence_covers_interior}
gives \(n_1\in\N\) such that
\(x\in K_n \qquad \forall n\ge n_1.\)
Hence \(A_n\neq\varnothing\) for every \(n\ge n_1\).

Let \(n\ge\max\{n_0,n_1\}\). By \textup{(ii)}, \(\rho_{\cK}\) is bounded
above on \(S_X\cap K_n\). Since \(A_n\neq\varnothing\), the first part of
the proof gives that \(\rho_{A_n}\) is bounded above on
\(\dom(\rho_{A_n})\). By Lemma~\ref{lem:norm_base_radial_profile}, \(A_n\)
admits a norm-base.

Thus \(A_n=\cK\cap K_n\) admits a norm-base for every sufficiently large
\(n\). By Proposition~\ref{prop:regularity_inner_approx_characterization},
\(\cK\) is regular.
\end{proof}

\noindent The preceding distinction is global versus local.

\begin{remark}
\label{rem:global_vs_local_radial_profile}
Corollary~\ref{cor:lipschitz_radial_profile_norm_base} is global: it implies
that \(\cK\) admits a norm-base. Proposition~\ref{prop:regularity_bounded_radial_profile_characterization}
is local: it only requires boundedness of \(\rho_{\cK}\) on the sections
\(S_X\cap K_n\) of an inner approximating sequence of \(\cone(\cK)\).
\end{remark}

\noindent In finite dimension, upper semicontinuity yields regularity.

\begin{proposition}
\label{prop:regularity_finite_dimension}
Assume that \(X\) is finite-dimensional. Let \(\cK\subset X\) be a co-radiant set,
and put \(K:=\cone(\cK)\). If \(K\) is solid and \(\rho_{\cK}\) is upper semicontinuous
on \(\dom(\rho_{\cK})\), then \(\cK\) is regular.
\end{proposition}

\begin{proof}
Set
\(K_n:=K^{-1/n} \quad(n\in\N).\)
By Proposition~\ref{prop:canonical_inner_approximating_sequence},
\((K_n)\) is a closed inner approximating sequence for \(K\). Since \(K\) is
solid, Lemma~\ref{lem:basic_properties_interior_reductions} gives
\(n_0\in\N\) such that
\(K_n\setminus\{0_X\}\neq\varnothing \qquad \forall n\ge n_0.\)

Fix \(n\ge n_0\). Since \(K_n\) is a closed cone, \(S_X\cap K_n\) is closed
in \(S_X\). Since \(X\) is finite-dimensional, \(S_X\) is compact. Hence
\(S_X\cap K_n\) is compact. Moreover,
\(S_X\cap K_n\subset S_X\cap K=\dom(\rho_{\cK}),\)
because \(K_n\subset K=\cone(\cK)\). Therefore \(\rho_{\cK}\) is finite-valued
and upper semicontinuous on the compact set \(S_X\cap K_n\). Hence
\(\rho_{\cK}\) is bounded above on \(S_X\cap K_n\).

Thus the boundedness condition in
Proposition~\ref{prop:regularity_bounded_radial_profile_characterization}
holds for every \(n\ge n_0\). Therefore \(\cK\) is regular.
\end{proof}

\noindent We return to the examples: the first three are regular without norm-bases. The fourth is not regular.

\noindent The translated quadrant is regular.

\begin{example}[Example~\ref{ex:K1-no-norm-base}, continued]
\label{ex:K1-regularity}
The set \(\cK_1=(1,1)+\Rp^2\) is regular.
\end{example}

\begin{proof}
By Example~\ref{ex:K1-radial-profile},
\(\dom(\rho_{\cK_1})=\{(a,b)\in S_X:\ a>0,\ b>0\}\) and
\(\rho_{\cK_1}(a,b)=\max\{1/a,1/b\}\). Hence \(\rho_{\cK_1}\) is continuous
on its domain. Also, by Example~\ref{ex:K1-no-norm-base},
\(\cone(\cK_1)=\R_{++}^2\cup\{0_X\}\), which is solid. Since \(X=\R^2\),
Proposition~\ref{prop:regularity_finite_dimension} gives the result.
\end{proof}

\noindent The hyperbolic epigraph is regular for the same reason.

\begin{example}[Example~\ref{ex:K2-no-norm-base}, continued]
\label{ex:K2-regularity}
The set \(\cK_2=\{(x,y)\in\R^2:\ x>0,\ y\ge1/x\}\) is regular.
\end{example}

\begin{proof}
By Examples~\ref{ex:K2-no-norm-base} and~\ref{ex:K2-radial-profile},
\(\cone(\cK_2)=\R_{++}^2\cup\{0_X\}\) is solid and
\(\rho_{\cK_2}(a,b)=1/\sqrt{ab}\) is continuous on
\(\dom(\rho_{\cK_2})=\{(a,b)\in S_X:\ a>0,\ b>0\}\). Since \(X=\R^2\),
Proposition~\ref{prop:regularity_finite_dimension} applies.
\end{proof}

\noindent The Bishop--Phelps type example is also regular.

\begin{example}[Example~\ref{ex:K3-no-norm-base}, continued]
\label{ex:K3-regularity}
Let \(X=\R^2\) be Euclidean, let \(f\in X^*\setminus\{0\}\),
\(0<\alpha<\|f\|_*\), and \(\lambda>0\). Then
\(\cK_3=\{x\in X:f(x)-\alpha\|x\|\ge\lambda\}\)
is regular.
\end{example}

\begin{proof}
Let \(v\in S_X\). For \(\mu>0\), we have
\[
\mu v\in\cK_3
\quad\Longleftrightarrow\quad
\mu\bigl(f(v)-\alpha\bigr)\ge\lambda.
\]
Hence \(\Rp v\cap\cK_3\neq\varnothing\) if and only if \(f(v)>\alpha\), and in that
case
\(\rho_{\cK_3}(v)=\frac{\lambda}{f(v)-\alpha}.\)
Therefore,
\(\dom(\rho_{\cK_3})=\{v\in S_X:f(v)>\alpha\},\)
and \(\rho_{\cK_3}\) is continuous on its domain.

Moreover, by Example~\ref{ex:K3-no-norm-base},
\[
\cone(\cK_3)=\intt C(f,\alpha)\cup\{0_X\},
\qquad
C(f,\alpha):=\{x\in X:f(x)-\alpha\|x\|\ge0\},
\]
and this cone is solid. Since \(X=\R^2\), Proposition~\ref{prop:regularity_finite_dimension}
applies.
\end{proof}

\noindent The last example shows that regularity is not automatic.

\begin{example}[Example~\ref{ex:K4-no-norm-base}, continued]
\label{ex:K4-nonregularity}
The set \(\cK_4\) satisfies \(\cone(\cK_4)=\Rp^2\), but it is not regular.
\end{example}

\begin{proof}
Set \(K:=\cone(\cK_4)=\Rp^2\). Let
\(\Gamma:=\{u_s:\ s\in[1,2]\}\). Since \(\Gamma\) is compact and
\(\Gamma\subset\intt K\cap S_X\), there exists \(\varepsilon>0\) such that
\(\Gamma\subset K^{-\varepsilon}\cap S_X\). As \(K^{-\varepsilon}\) is a
cone, \(\Rp u_s\subset K^{-\varepsilon}\) for every \(s\in[1,2]\).

Set \(\cK_4^{-\varepsilon}:=\cK_4\cap K^{-\varepsilon}\). By
Example~\ref{ex:K4-radial-profile},
\[
\rho_{\cK_4^{-\varepsilon}}(u_s)
=
\rho_{\cK_4}(u_s)
=
\phi(s),
\qquad s\in[1,2].
\]
The function \(\phi\) is unbounded on \([1,2]\): for \(q\ge2\),
\(s_q=(q+1)/q\in[1,2]\) and \(\phi(s_q)=q\). Hence
\(\rho_{\cK_4^{-\varepsilon}}\) is unbounded above on its domain. By
Remark~\ref{rem:norm_base_bounded_radial_profile},
\(\cK_4^{-\varepsilon}\) has no norm-base. Since
\(\cK_4^{-\varepsilon}\neq\varnothing\), \(\cK_4\) is not regular.
\end{proof}

\subsection{Radial Scalarization and Approximate Efficiency}
\label{subsec:radial_scalarization_approximate_efficiency}
We begin this subsection by introducing terminology related to approximate efficiency.

Let \(Y\) be a normed space, \(\Omega\subset X\) a nonempty set, \(f:\Omega\to Y\), and
\(\cK\subset Y\) a co-radiant set. For \(\varepsilon>0\), define
\(\cK(\varepsilon):=\varepsilon\cK\). A point \(\bar x\in\Omega\) is said to be
\(\cK(\varepsilon)\)-approximately efficient if
\(f(\Omega)\cap\bigl(f(\bar x)-\cK(\varepsilon)\bigr)=\varnothing.\)
Equivalently,
\(f(\bar x)-f(x)\notin\cK(\varepsilon) \qquad \forall x\in\Omega.\)

Next, using the radial profile function, we introduce a function $\phi_{\mathcal K}$ associated with each co-radiant set $\mathcal K$, which plays the role of a characteristic-type function for $\mathcal K$. This construction will allow us to derive scalarization results for approximate efficiency problems.

Define \(\phi_{\cK}:Y\to[0,+\infty]\) by
\[
\phi_{\cK}(y):=
\begin{cases}
0, & \text{if } y=0_Y,\\[1ex]
\dfrac{\|y\|}{\rho_{\cK}(y/\|y\|)}, & \text{if } y\neq0_Y,
\end{cases}
\]
with the conventions
\[
\frac{a}{+\infty}:=0,
\qquad
\frac{a}{0}:=+\infty,
\qquad
a(+\infty):=+\infty
\qquad (a>0).
\]
The indeterminate product \(0\cdot(+\infty)\) will not arise below, since scalar multiplication of \(+\infty\) is used only with strictly positive scalars.
Thus,
\[
\phi_{\cK}(ty)=t\phi_{\cK}(y)
\qquad
\forall\,t>0,\quad \forall\,y\in Y,
\]
while \(\phi_{\cK}(0_Y)=0\). Therefore, for every \(\varepsilon>0\),
\[
\rho_{\cK(\varepsilon)}(u)=\varepsilon\rho_{\cK}(u)
\qquad (u\in S_Y),
\]
and hence \(\phi_{\cK(\varepsilon)}(y)=\varepsilon^{-1}\phi_{\cK}(y)\) for every \(y\in Y\).

\noindent The following inclusion provides a first one-sided scalarization.
\begin{lemma}
\label{lem:radial_scalarization_basic_inclusion}
For every \(\varepsilon>0\),
\(\cK(\varepsilon) \subset \{y\in Y:\phi_{\cK}(y)\ge\varepsilon\}.\)
\end{lemma}

\begin{proof}
Let \(y\in\cK(\varepsilon)\). Then \(y=\varepsilon k\) for some
\(k\in\cK\). Since \(0_Y\notin\cK\), \(k\neq0_Y\). Set
\(u:=k/\|k\|\). Since \(k=\|k\|u\in\cK\),
\[
\rho_{\cK}(u)\le\|k\|<+\infty.
\]
If \(\rho_{\cK}(u)=0\), then \(\phi_{\cK}(k)=+\infty\). If
\(\rho_{\cK}(u)>0\), then
\[
\phi_{\cK}(k)
=
\frac{\|k\|}{\rho_{\cK}(u)}
\ge1.
\]
Thus \(\phi_{\cK}(k)\ge1\) in either case, and positive homogeneity yields
\[
\phi_{\cK}(y)
=
\varepsilon\phi_{\cK}(k)
\ge\varepsilon.
\]
\end{proof}

\noindent This immediately yields a sufficient condition for approximate efficiency.

\begin{proposition}
\label{prop:sufficient_radial_scalarization_approx_efficiency}
Let \(\varepsilon>0\) and \(\bar x\in\Omega\). If
\(\phi_{\cK}\bigl(f(\bar x)-f(x)\bigr)<\varepsilon \qquad \forall x\in\Omega,\)
then \(\bar x\) is \(\cK(\varepsilon)\)-approximately efficient.
\end{proposition}

\begin{proof}
Fix \(x\in\Omega\) and set
\(y:=f(\bar x)-f(x).\)
If \(y\in\cK(\varepsilon)\), then
Lemma~\ref{lem:radial_scalarization_basic_inclusion} gives
\(\phi_{\cK}(y)\ge\varepsilon,\)
contrary to the assumption. Hence
\(f(\bar x)-f(x)\notin\cK(\varepsilon) \qquad \forall x\in\Omega.\)
This proves the result.
\end{proof}

\noindent The associated residual is obtained by taking the worst scalar violation. Define
\[
\Psi_{\cK}(\bar x)
:=
\sup_{x\in\Omega}
\phi_{\cK}\bigl(f(\bar x)-f(x)\bigr),
\qquad \bar x\in\Omega.
\]
Then
\[
\Psi_{\cK}(\bar x)<\varepsilon
\quad\Longrightarrow\quad
\bar x
\text{ is }
\cK(\varepsilon)\text{-approximately efficient}.
\]

\noindent Exactness requires that the profile reconstructs the co-radiant set.

\begin{definition}
\label{def:radial_reconstruction}
The co-radiant set \(\cK\) is radially reconstructed by \(\phi_{\cK}\) if
\(\cK = \{y\in Y:\phi_{\cK}(y)\ge1\}.\)
\end{definition}

\noindent Under radial reconstruction, the scaled sets are exactly scalar level sets.

\begin{lemma}
\label{lem:radial_reconstruction_scaled}
If \(\cK\) is radially reconstructed by \(\phi_{\cK}\), then, for every
\(\varepsilon>0\),
\(\cK(\varepsilon) = \{y\in Y:\phi_{\cK}(y)\ge\varepsilon\}.\)
\end{lemma}

\begin{proof}
The inclusion \(\subset\) follows from
Lemma~\ref{lem:radial_scalarization_basic_inclusion}. Conversely, let
\(\phi_{\cK}(y)\ge\varepsilon\). Since \(\varepsilon>0\), necessarily
\(y\neq0_Y\). By positive homogeneity,
\[
\phi_{\cK}\left(\frac{1}{\varepsilon}y\right)
=
\frac{1}{\varepsilon}\phi_{\cK}(y)
\ge1.
\]
By radial reconstruction,
\(\frac{1}{\varepsilon}y\in\cK.\)
Hence \(y\in\varepsilon\cK=\cK(\varepsilon)\).
\end{proof}

\noindent This gives a pointwise exact scalarization of approximate efficiency.

\begin{proposition}
\label{prop:exact_radial_scalarization_approx_efficiency}
Assume that \(\cK\) is radially reconstructed by \(\phi_{\cK}\). Let
\(\varepsilon>0\) and \(\bar x\in\Omega\). Then \(\bar x\) is
\(\cK(\varepsilon)\)-approximately efficient if and only if
\[\phi_{\cK}\bigl(f(\bar x)-f(x)\bigr)<\varepsilon \qquad \forall x\in\Omega.\]
\end{proposition}

\begin{proof}
By definition, \(\bar x\) is \(\cK(\varepsilon)\)-approximately efficient if
and only if
\(f(\bar x)-f(x)\notin\cK(\varepsilon) \quad \forall x\in\Omega.\)
By Lemma~\ref{lem:radial_reconstruction_scaled}, this is equivalent to
\(\phi_{\cK}\bigl(f(\bar x)-f(x)\bigr)<\varepsilon \quad \forall x\in\Omega.\)
\end{proof}

\noindent The residual gives a strict test only under a uniform margin.

\begin{corollary}
\label{cor:radial_scalarization_with_strict_margin}
Assume that \(\cK\) is radially reconstructed by \(\phi_{\cK}\). Let
\(\varepsilon>0\) and \(\bar x\in\Omega\). If
\(\Psi_{\cK}(\bar x)<\varepsilon,\)
then \(\bar x\) is \(\cK(\varepsilon)\)-approximately efficient. Conversely,
if \(\bar x\) is \(\cK(\varepsilon)\)-approximately efficient, then
\[\phi_{\cK}\bigl(f(\bar x)-f(x)\bigr)<\varepsilon \qquad \forall x\in\Omega,\]
and therefore
\(\Psi_{\cK}(\bar x)\le\varepsilon.\)
\end{corollary}

\begin{proof}
The first assertion follows from
Proposition~\ref{prop:sufficient_radial_scalarization_approx_efficiency}.
The converse follows from
Proposition~\ref{prop:exact_radial_scalarization_approx_efficiency}. Taking
suprema gives \(\Psi_{\cK}(\bar x)\le\varepsilon\).
\end{proof}

\noindent The final remark records the gap between pointwise and uniform strictness.

\begin{remark}
\label{rem:supremum_gap_radial_scalarization}
The pointwise strict condition
\(\phi_{\cK}\bigl(f(\bar x)-f(x)\bigr)<\varepsilon \qquad \forall x\in\Omega\)
does not imply, in general, that
\(\Psi_{\cK}(\bar x)<\varepsilon.\)
It only implies \(\Psi_{\cK}(\bar x)\le\varepsilon\). Thus the exact
scalarization is pointwise. The residual \(\Psi_{\cK}\) gives a strict
scalar test only when a uniform margin below \(\varepsilon\) is available.
\end{remark}

\begin{example}
\label{ex:radial_scalarization_uniform_gap}
Let \(Y=\R^2\) be Euclidean and define
\(\cK:= \{y=(y_1,y_2)\in\R^2:\ y_2-|y_1|\ge1\}.\)
Then \(\cK\) is co-radiant. Moreover, for \(u=(u_1,u_2)\in S_Y\),
\[
\rho_{\cK}(u)
=
\begin{cases}
\dfrac{1}{u_2-|u_1|}, & \text{if } u_2>|u_1|,\\[1.2ex]
+\infty, & \text{if } u_2\le |u_1|.
\end{cases}
\]
Consequently,
\(\phi_{\cK}(y) = (y_2-|y_1|)_+, \qquad y=(y_1,y_2)\in\R^2.\)
In particular,
\(\cK=\{y\in\R^2:\phi_{\cK}(y)\ge1\},\)
so \(\cK\) is radially reconstructed by \(\phi_{\cK}\).

\medskip

Now let \(\Omega=[0,1)\), define \(f:\Omega\to\R^2\) by
\(f(t):=(0,-t), \qquad t\in[0,1),\)
and take \(\bar x=0\). Then, for every \(t\in[0,1)\),
\(f(\bar x)-f(t)=(0,t),\)
and therefore
\(\phi_{\cK}\bigl(f(\bar x)-f(t)\bigr)=t<1.\)
By Proposition~\ref{prop:exact_radial_scalarization_approx_efficiency},
\(\bar x\) is \(\cK(1)\)-approximately efficient.

However,
\[
\Psi_{\cK}(\bar x)
=
\sup_{t\in[0,1)}
\phi_{\cK}\bigl(f(\bar x)-f(t)\bigr)
=
\sup_{t\in[0,1)}t
=
1.
\]
Thus \(\bar x\) is \(\cK(1)\)-approximately efficient, but the strict residual test
\(\Psi_{\cK}(\bar x)<1\) fails. This illustrates the gap between the pointwise exact
scalarization and the uniform strict test given by the residual \(\Psi_{\cK}\).
\end{example}

\section{A Separation Theorem for Regular Co-Radiant Sets}
\label{sec:main_separation_theorem}
In this section we establish the separation result for regular co-radiant sets. In contrast with the classical setting, where the existence of a norm base plays a central role, here such an assumption is replaced by the weaker notion of regularity. The proof relies on the conical framework developed in the previous section, in particular on the stability properties of the  SSP along inner approximating sequences. This allows us to reduce the general separation problem to an appropriate approximating family of cones, thereby extending the classical separation theory for co-radiant sets.

The result shows that, under the SSP for the pair of cones generated by the sets, regularity of the second co-radiant set is sufficient to guarantee the existence of a hyperbolic-type separating functional. More precisely, the theorem provides a separation between the closed the convex hull of $\mathcal{C}$ and suitable outer enlargements of $\mathcal{K}$, expressed in terms of the 
Bishop--Phelps type scalar function $f(x)-\alpha\|x\|$.

\begin{theorem}[Regular co-radiant separation theorem]
\label{teorema:separacion_coradiantes_regular}
Let \(\cC,\cK\subset X\) be co-radiant sets such that
\[
d_{\cC}d_{\cK}>0,\qquad
\cK \text{ is regular},\qquad
\clco(\cC)\cap\intt(\cK)\neq\varnothing.
\]
Assume that \((\cone(\cC),\cone(\cK))\) verifies the SSP. Then there exist
\(0<\alpha_1<\alpha_2<1\) and \(f\in S_{X^*}\) such that:
\begin{enumerate}[label=\textup{(\alph*)}]
\item
\(
\lambda:=
\inf\{f(y)-\alpha\|y\|:\ y\in\cC,\ \alpha\in(\alpha_1,\alpha_2)\}>0.
\)
\item There exists \(\eta_0>0\) such that, for every \(\eta\in(0,\eta_0)\),
\[f(y')-\alpha\|y'\| < \frac{\eta}{\eta_0}\lambda < \lambda < f(y)-\alpha\|y\|\]
for every \(\alpha\in(\alpha_1,\alpha_2)\), every
\(y'\in X\setminus\intt(\cK(\eta))\), and every \(y\in\clco(\cC)\).
\end{enumerate}
\end{theorem}

\begin{proof}
Set
\(H:=\cone(\cK).\)
Since \(\clco(\cC)\cap\intt(\cK)\neq\varnothing\), choose
\(x_0\in\clco(\cC)\cap\intt(\cK).\)
Then \(x_0\neq0_X\) and \(x_0\in\intt H\). In particular, \(H\) is solid.

Let
\(H_n:=H^{-1/n} \quad(n\in\N).\)
By Proposition~\ref{prop:canonical_inner_approximating_sequence},
\((H_n)\) is a closed inner approximating sequence for \(H\).

By Theorem~\ref{teo:ssp_inner_approximating_sequence}, applied to the pair
\((\cone(\cC),H)\), there exists \(n_1\in\N\) such that
\((\cone(\cC),H_n)\text{ verifies the SSP} \quad \forall n\ge n_1.\)
By Proposition~\ref{prop:inner_approximating_sequence_covers_interior}, there
exists \(n_2\in\N\) such that
\(x_0\in H_n \quad \forall n\ge n_2.\)
Since \(\cK\) is regular and \(x_0\in\cK\cap H_n\) for \(n\ge n_2\), the set
\(\cK\cap H_n = \cK\cap H^{-1/n} = \cK^{-1/n}\)
admits a norm-base for every \(n\ge n_2\).
Choose
\(n\ge\max\{n_1,n_2\}\)
and set
\(\cK':=\cK\cap H_n.\)
Since \(\cK'\subset\cK\), we have
\(
d_{\cK'}\ge d_{\cK}>0.
\)
Moreover, since \(\cK'\) admits a norm-base,
\[
0<d_{\cK'}\le \mathcal I_{\cK'}<+\infty.
\]
Then \(\cK'\) admits a norm-base, \(x_0\in\cK'\), and
\((\cone(\cC),H_n)\text{ verifies the SSP}.\)
By Proposition~\ref{prop:cone_of_intersection_with_subcone},
\(\cone(\cK')=\cone(\cK\cap H_n)=H_n.\)
Hence
\((\cone(\cC),\cone(\cK'))\) verifies the SSP.
Moreover,
\(\clco(\cC)\cap\cK'\neq\varnothing,\)
because \(x_0\in\clco(\cC)\cap\cK'\).

Therefore the hypotheses of \cite[Theorem 4.6]{GarciaCastanoMelguizo2025JGO} are satisfied
by the pair \((\cC,\cK')\). Thus there exist
\(0<\alpha_1<\alpha_2<1\), \(f\in S_{X^*}\), and
\(\lambda:= \inf\{f(y)-\alpha\|y\|:\ y\in\cC,\ \alpha\in(\alpha_1,\alpha_2)\}>0\)
such that, with
\(\bar\eta_0:=\frac{\lambda}{\mathcal I_{\cK'}},\)
one has, for every \(\eta\in(0,\bar\eta_0)\),
\(f(y')-\alpha\|y'\|<\lambda<f(y)-\alpha\|y\|\)
for every \(\alpha\in(\alpha_1,\alpha_2)\), every
\(y'\in X\setminus\intt(\cK'(\eta))\), and every \(y\in\clco(\cC)\).

Fix
\(\eta_0\in(0,\bar\eta_0).\)
Since \(\cK'\subset\cK\), for every \(\eta>0\),
\(\cK'(\eta)\subset\cK(\eta),\)
and hence
\(X\setminus\intt(\cK(\eta)) \subset X\setminus\intt(\cK'(\eta)).\)
It follows that, for every \(\eta\in(0,\bar\eta_0)\),
\(f(y')-\alpha\|y'\|<\lambda<f(y)-\alpha\|y\|\)
for every \(\alpha\in(\alpha_1,\alpha_2)\), every
\(y'\in X\setminus\intt(\cK(\eta))\), and every \(y\in\clco(\cC)\).

Let now \(\eta\in(0,\eta_0)\), \(\alpha\in(\alpha_1,\alpha_2)\), and
\(y'\in X\setminus\intt(\cK(\eta)).\)
Set
\(t:=\frac{\eta_0}{\eta}>1.\)
Since scalar multiplication by \(t\) is a homeomorphism,
\[
t\,\intt(\cK(\eta))
=
\intt(t\,\cK(\eta))
=
\intt(\cK(t\eta))
=
\intt(\cK(\eta_0)).
\]
Thus
\(ty'\in X\setminus\intt(\cK(\eta_0)).\)
Since \(\eta_0<\bar\eta_0\), the preceding estimate applied with \(\eta_0\)
gives
\(f(ty')-\alpha\|ty'\|<\lambda.\)
By linearity of \(f\) and homogeneity of the norm,
\(t\bigl(f(y')-\alpha\|y'\|\bigr)<\lambda.\)
Therefore
\(f(y')-\alpha\|y'\| < \frac{\lambda}{t} = \frac{\eta}{\eta_0}\lambda.\)
for every \(y\in\clco(\cC)\),
\(\lambda<f(y)-\alpha\|y\|.\)
Hence
\(f(y')-\alpha\|y'\| < \frac{\eta}{\eta_0}\lambda < \lambda < f(y)-\alpha\|y\|.\)
This proves \textup{(a)} and \textup{(b)}.
\end{proof}

\begin{remark}
It is worth noting that the separation results obtained in \cite{GarciaCastanoMelguizo2025JGO} are recovered as a particular case of the present theorem. Indeed, since every co-radiant set admitting a norm base is regular, the assumptions of Theorem~\ref{teorema:separacion_coradiantes_regular} are satisfied in the setting of \cite{GarciaCastanoMelguizo2025JGO}. In particular, the corresponding separation results for co-radiant sets with norm bases follow by specializing the present framework, thereby showing that the classical theory can be fully embedded into the regularity-based approach developed in this work.
\end{remark}

For later use, we reformulate the preceding separation theorem in the following form, which is tailored to the variational error-bound arguments developed in
Section~\ref{sec:variational_consequence_regular_separation}.

\begin{corollary}
\label{cor:separacion_coradiantes_regular_con_theta}
Let \(\cC,\cK\subset X\) be co-radiant sets such that
\[
d_{\cC}d_{\cK}>0,\qquad
\cK \text{ is regular},\qquad
\clco(\cC)\cap\intt(\cK)\neq\varnothing.
\]
Consider the following assertions:
\begin{enumerate}[label=\textup{(\roman*)}]
\item \((\cone(\cC),\cone(\cK))\) verifies the SSP.
\item There exist \(0<\alpha_1<\alpha_2<1\) and \(f\in S_{X^*}\) such that:
\begin{enumerate}[label=\textup{(\alph*)}]
\item
\(
\lambda:=
\inf\{f(y)-\alpha\|y\|:\ y\in\cC,\ \alpha\in(\alpha_1,\alpha_2)\}>0.
\)
\item There exists \(\eta_0>0\) such that, for every \(0<\theta<1\) and every
\(\eta\in(0,\theta\eta_0)\),
\[f(y')-\alpha\|y'\| < \lambda\theta < \lambda < f(y)-\alpha\|y\|\]
for every \(\alpha\in(\alpha_1,\alpha_2)\), every
\(y'\in X\setminus\intt(\cK(\eta))\), and every \(y\in\clco(\cC)\).
\end{enumerate}
\end{enumerate}
Then \textup{(i)} implies \textup{(ii)}.
\end{corollary}

\begin{proof}
Assume \textup{(i)}. By Theorem~\ref{teorema:separacion_coradiantes_regular},
there exist \(0<\alpha_1<\alpha_2<1\), \(f\in S_{X^*}\), \(\lambda>0\), and
\(\eta_0>0\) such that, for every \(\eta\in(0,\eta_0)\),
\(f(y')-\alpha\|y'\| < \frac{\eta}{\eta_0}\lambda < \lambda < f(y)-\alpha\|y\|\)
for every \(\alpha\in(\alpha_1,\alpha_2)\), every
\(y'\in X\setminus\intt(\cK(\eta))\), and every \(y\in\clco(\cC)\).

Let \(0<\theta<1\) and \(\eta\in(0,\theta\eta_0)\). Then
\(0<\frac{\eta}{\eta_0}<\theta.\)
Since \(\lambda>0\),
\(\frac{\eta}{\eta_0}\lambda<\theta\lambda.\)
Therefore
\(f(y')-\alpha\|y'\| < \lambda\theta < \lambda < f(y)-\alpha\|y\|\)
for every \(\alpha\in(\alpha_1,\alpha_2)\), every
\(y'\in X\setminus\intt(\cK(\eta))\), and every \(y\in\clco(\cC)\).
Thus \textup{(ii)} holds.
\end{proof}

\section{From Regular Separation to Variational Error Bounds}
\label{sec:variational_consequence_regular_separation}
We extract from Theorem~\ref{teorema:separacion_coradiantes_regular} a scalar residual and a radial depth estimate. This gives an error-bound interpretation of
the regular separation theorem.

Assume throughout this section that the hypotheses of
Corollary~\ref{cor:separacion_coradiantes_regular_con_theta} hold. Thus there
exist \(0<\alpha_1<\alpha_2<1\), \(f\in S_{X^*}\), \(\lambda>0\), and
\(\eta_0>0\) such that, for every \(0<\theta<1\), every
\(\eta\in(0,\theta\eta_0)\), and every \(\alpha\in(\alpha_1,\alpha_2)\),
\(f(y')-\alpha\|y'\| < \lambda\theta < \lambda < f(y)-\alpha\|y\|\)
for every \(y'\in X\setminus\intt(\cK(\eta))\) and every
\(y\in\clco(\cC)\).

For \(\alpha\in(\alpha_1,\alpha_2)\), set
\[
p_\alpha(y):=f(y)-\alpha\|y\|,
\qquad
G_\alpha(y):=(\lambda-p_\alpha(y))_+,
\qquad y\in X.
\]
The function \(G_\alpha\) is the residual associated with the separating
level \(\lambda\).

\subsection{The Radial Depth Function}
\label{subsec:radial_depth_function}
\begin{definition}
\label{def:radial_depth_function}
Let \(\cK\subset X\) be co-radiant. The radial depth function of \(\cK\) is
\[
\delta_{\cK}:X\setminus\{0_X\}\to[0,+\infty],
\qquad
\delta_{\cK}(y):=\sup\{\eta>0:\ y\in\cK(\eta)\},
\]
with the convention \(\sup\varnothing=0\).
\end{definition}
\begin{lemma}
\label{lem:interval_depth}
Let \(\cK\subset X\) be co-radiant and let \(y\in X\). Then
\(J(y):=\{\eta>0:\ y\in\cK(\eta)\}\)
is an interval of one of the following forms:
\[
\varnothing,\qquad
(0,\bar\eta),\qquad
(0,\bar\eta],\qquad
(0,+\infty),
\]
where \(\bar\eta\in(0,+\infty)\).
\end{lemma}

\begin{proof}
Let \(\eta_2\in J(y)\) and \(0<\eta_1<\eta_2\). Writing
\(y=\eta_2k\) with \(k\in\cK\), co-radiance yields
\[
\frac{\eta_2}{\eta_1}k\in\cK,
\qquad
y=\eta_1\left(\frac{\eta_2}{\eta_1}k\right)\in\eta_1\cK.
\]
Hence \(J(y)\) is downward closed in \(\mathbb R_{++}\). Thus
\(J(y)=\varnothing\) or \(J(y)=\mathbb R_{++}\); otherwise, choosing
\(\eta_0\notin J(y)\) gives \(J(y)\subset(0,\eta_0)\), and therefore
\[
\bar\eta:=\sup J(y)\in(0,+\infty).
\]
Downward closedness then implies
\[
(0,\bar\eta)\subset J(y)\subset(0,\bar\eta],
\]
so \(J(y)=(0,\bar\eta)\) or \(J(y)=(0,\bar\eta]\).
\end{proof}

\begin{remark}
\label{rem:interpretacion_deltaK}
For every \(y\in X\setminus\{0_X\}\) and every \(\eta>0\),
\(\delta_{\cK}(y)>\eta \quad\Longrightarrow\quad y\in\cK(\eta),\)
whereas
\(y\in\cK(\eta) \quad\Longrightarrow\quad \delta_{\cK}(y)\ge\eta.\)
Thus only the endpoint \(\eta=\delta_{\cK}(y)\) may be ambiguous.
If \(\cK\) is closed, the endpoint is attained whenever
\(0<\delta_{\cK}(y)<+\infty\). If \(\cK\) is open, the endpoint is not
attained whenever \(0<\delta_{\cK}(y)<+\infty\).
\end{remark}

\begin{lemma}
\label{lem:basic_properties_deltaK}
Let \(\cK\subset X\) be co-radiant. Then:
\begin{enumerate}[label=\textup{(\roman*)}]
\item for every \(y\in X\setminus\{0_X\}\) and every \(t>0\),
\(\delta_{\cK}(ty)=t\,\delta_{\cK}(y);\)
\item if \(\cK_1,\cK_2\subset X\) are co-radiant and
\(\cK_1\subset\cK_2\), then
\(\delta_{\cK_1}(y)\le\delta_{\cK_2}(y) \quad \forall y\in X\setminus\{0_X\};\)
\item for every \(y\in X\setminus\{0_X\}\),
\[
\delta_{\cK}(y)>1
\quad\Longrightarrow\quad
y\in\cK
\quad\Longrightarrow\quad
\delta_{\cK}(y)\ge1.
\]
Equivalently, for every \(\eta>0\),
\[
\delta_{\cK}(y)>\eta
\quad\Longrightarrow\quad
y\in\cK(\eta),
\qquad
y\in\cK(\eta)
\quad\Longrightarrow\quad
\delta_{\cK}(y)\ge\eta.
\]
\end{enumerate}
\end{lemma}

\begin{proof}
\textup{(i)} For every \(\eta>0\),
\(y\in\cK(\eta) \quad\Longleftrightarrow\quad ty\in t\cK(\eta).\)
Since
\(t\cK(\eta)=t(\eta\cK)=(t\eta)\cK=\cK(t\eta),\)
we have
\(\{\mu>0:\ ty\in\cK(\mu)\} = \{t\eta:\ \eta>0,\ y\in\cK(\eta)\}.\)
Taking suprema gives the assertion.

\textup{(ii)} If \(\cK_1\subset\cK_2\), then
\(\cK_1(\eta)=\eta\cK_1\subset\eta\cK_2=\cK_2(\eta) \qquad \forall\eta>0.\)
The conclusion follows by taking suprema.

\textup{(iii)} Let \(\eta>0\). If \(\delta_{\cK}(y)>\eta\), then there exists
\(\eta'>\eta\) such that \(y\in\cK(\eta')\). Since \(0<\eta<\eta'\) and
\(\cK\) is co-radiant,
\(\cK(\eta')\subset\cK(\eta).\)
Hence \(y\in\cK(\eta)\). Conversely, if \(y\in\cK(\eta)\), then \(\eta\)
belongs to the defining set of \(\delta_{\cK}(y)\), and so
\(\delta_{\cK}(y)\ge\eta\). Taking \(\eta=1\) gives the first implication.
\end{proof}

\subsection{Separation-induced Error Bounds and Variational Consequences}
\label{subsec:separation_induced_error_bounds}

We now convert the separation inequality into a radial depth estimate.

\begin{corollary}
\label{cor:separacion_coradiantes_regular_theta_delta}
For every \(\alpha\in(\alpha_1,\alpha_2)\), every \(0<\theta<1\), and every
\(y\in X\setminus\{0_X\}\),
\[
p_\alpha(y)\ge\lambda\theta
\quad\Longrightarrow\quad
\delta_{\cK}(y)\ge\theta\eta_0.
\]
Equivalently,
\[
\delta_{\cK}(y)<\theta\eta_0
\quad\Longrightarrow\quad
p_\alpha(y)<\lambda\theta.
\]
In particular,
\(p_\alpha(y)\ge\lambda \quad\Longrightarrow\quad \delta_{\cK}(y)\ge\eta_0.\)
\end{corollary}

\begin{proof}
Fix \(\alpha\in(\alpha_1,\alpha_2)\), \(0<\theta<1\), and
\(y\in X\setminus\{0_X\}\). Assume that \(p_\alpha(y)\ge\lambda\theta\).

Suppose that \(\delta_{\cK}(y)<\theta\eta_0\). Choose
\(\delta_{\cK}(y)<\eta<\theta\eta_0.\)
Then \(y\notin\cK(\eta)\), and therefore
\(y\notin\intt(\cK(\eta))\). By
Corollary~\ref{cor:separacion_coradiantes_regular_con_theta},
\(p_\alpha(y)<\lambda\theta,\)
a contradiction. Hence \(\delta_{\cK}(y)\ge\theta\eta_0\).

The equivalent formulation is the contrapositive. If \(p_\alpha(y)\ge\lambda\),
then \(p_\alpha(y)\ge\lambda\theta\) for every \(0<\theta<1\). Hence
\(\delta_{\cK}(y)\ge\theta\eta_0 \qquad \forall\,0<\theta<1.\)
Letting \(\theta\uparrow1\), we obtain
\(\delta_{\cK}(y)\ge\eta_0.\)
\end{proof}

\begin{corollary}
\label{cor:gap_zero_implies_depth}
For every \(\alpha\in(\alpha_1,\alpha_2)\) and every
\(y\in X\setminus\{0_X\}\),
\(G_\alpha(y)=0 \quad\Longrightarrow\quad \delta_{\cK}(y)\ge\eta_0.\)
\end{corollary}

\begin{proof}
If \(G_\alpha(y)=0\), then \(p_\alpha(y)\ge\lambda\). The result follows from
Corollary~\ref{cor:separacion_coradiantes_regular_theta_delta}.
\end{proof}

We next record the quantitative content of the regular separation. The residual
\(G_\alpha\) controls the radial-depth violation \((\eta_0-\delta_{\cK}(\cdot))_+\) and the distance to its zero-level set. Hence, it provides a certificate of approximate radial feasibility.

\begin{proposition}[Error bounds generated by the separation residual]
\label{prop:error_bounds_from_separation_residual}
Fix \(\alpha\in(\alpha_1,\alpha_2)\). The following error bounds hold.

\textup{(i)} For every \(y\in X\setminus\{0_X\}\),
\[
\bigl(\eta_0-\delta_{\cK}(y)\bigr)_+
\le
\frac{\eta_0}{\lambda}\,G_\alpha(y).
\]

\textup{(ii)} For every \(y\in X\),
\[
\dist(y,F_\alpha)
\le
\frac{2}{1-\alpha}G_\alpha(y),
\]
where
\[
F_\alpha:=\{y\in X:\ p_\alpha(y)\ge\lambda\}
=\{y\in X:\ G_\alpha(y)=0\}.
\]
Moreover, every \(y\in F_\alpha\) satisfies \(\delta_{\cK}(y)\ge\eta_0\).
\end{proposition}

\begin{proof}
We prove \textup{(i)}. Let \(y\in X\setminus\{0_X\}\). If \(p_\alpha(y)\ge\lambda\), then \(G_\alpha(y)=0\), and Corollary~\ref{cor:separacion_coradiantes_regular_theta_delta} gives \(\delta_{\cK}(y)\ge\eta_0\). Hence the estimate is immediate.

Assume that \(0<p_\alpha(y)<\lambda\), and set \(\theta:=p_\alpha(y)/\lambda\). Then \(0<\theta<1\), and Corollary~\ref{cor:separacion_coradiantes_regular_theta_delta} yields
\[
\delta_{\cK}(y)\ge\theta\eta_0
=
\frac{\eta_0}{\lambda}p_\alpha(y).
\]
Therefore
\[
\eta_0-\delta_{\cK}(y)
\le
\frac{\eta_0}{\lambda}\bigl(\lambda-p_\alpha(y)\bigr)
=
\frac{\eta_0}{\lambda}G_\alpha(y),
\]
and the estimate follows. Finally, if \(p_\alpha(y)\le0\), then \(G_\alpha(y)=\lambda-p_\alpha(y)\ge\lambda\), whence
\[
\frac{\eta_0}{\lambda}G_\alpha(y)
\ge
\eta_0
\ge
\bigl(\eta_0-\delta_{\cK}(y)\bigr)_+.
\]
This proves \textup{(i)}.

We prove \textup{(ii)}. Let \(y\in X\). If \(y\in F_\alpha\), the estimate is trivial. Assume that \(y\notin F_\alpha\). Then \(G_\alpha(y)=\lambda-p_\alpha(y)>0\). Since \(f\in S_{X^*}\) and \(\alpha<1\), there exists \(u\in S_X\) such that \(f(u)>(1+\alpha)/2\). Hence, for every \(t > 0\),
\[
\begin{aligned}
p_\alpha(y+tu)
&= f(y+tu)-\alpha\|y+tu\| \\
&\ge f(y)+tf(u)-\alpha(\|y\|+t) \\
&= p_\alpha(y)+t(f(u)-\alpha) \\
&> p_\alpha(y)+\frac{1-\alpha}{2}t.
\end{aligned}
\]
Taking \(t_0:=2G_\alpha(y)/(1-\alpha)\) and \(z:=y+t_0u\), we obtain
\[
p_\alpha(z)>p_\alpha(y)+G_\alpha(y)=\lambda.
\]
Thus \(z\in F_\alpha\), and
\[
\dist(y,F_\alpha)
\le
\|z-y\|
=
t_0
=
\frac{2}{1-\alpha}G_\alpha(y).
\]
This proves \textup{(ii)}.

Finally, let \(y\in F_\alpha\). Then \(p_\alpha(y)\ge\lambda>0\). Since \(p_\alpha(0_X)=0\), we have \(y\neq0_X\). By Corollary~\ref{cor:separacion_coradiantes_regular_theta_delta}, \(\delta_{\cK}(y)\ge\eta_0\).
\end{proof}

\begin{remark}
\label{rem:core_variational_consequence}
The quantitative content of Theorem \ref{teorema:separacion_coradiantes_regular} is twofold. First,
Proposition~\ref{prop:error_bounds_from_separation_residual}\textup{(i)} shows
that \(G_\alpha\) controls the radial-depth violation
\[
\bigl(\eta_0-\delta_{\cK}(y)\bigr)_+,
\qquad y\in X\setminus\{0_X\}.
\]
Second, Proposition~\ref{prop:error_bounds_from_separation_residual}\textup{(ii)}
shows that \(G_\alpha\) controls the distance to its zero set
\[
F_\alpha=\{y\in X:\ G_\alpha(y)=0\}.
\]
Moreover, Theorem~\ref{teorema:separacion_coradiantes_regular} yields
\(\cC\subset F_\alpha\), while
Proposition~\ref{prop:error_bounds_from_separation_residual} gives
\[
F_\alpha\subset
\{y\in X\setminus\{0_X\}:\delta_{\cK}(y)\ge\eta_0\}.
\]
Thus \(G_\alpha\) separates \(\cC\) from the forbidden radial region, controls
the distance to the separating zero set \(F_\alpha\), and provides an a
posteriori certificate of approximate radial feasibility.
\end{remark}

We now give two variational applications of the preceding estimates. Fix
\(\alpha\in(\alpha_1,\alpha_2)\) and set
\[
R_\alpha(y):=\frac{\eta_0}{\lambda}G_\alpha(y),
\qquad y\in X.
\]
By Proposition~\ref{prop:error_bounds_from_separation_residual}\textup{(i)},
\(R_\alpha\) controls the radial-depth violation and therefore provides a
certificate of approximate radial feasibility. This is applied first to
equilibrium problems, with variational inequalities as a particular case. The
metric error bound in
Proposition~\ref{prop:error_bounds_from_separation_residual}\textup{(ii)} is then
used to obtain an exact penalization principle for \(F_\alpha=\{G_\alpha=0\}\).

\subsubsection{Equilibrium Problems and Variational Inequalities with Radial Feasibility}
\label{subsubsec:equilibrium_vi_radial_feasibility}
We first apply the radial residual \(R_\alpha\) to equilibrium problems with an additional radial feasibility requirement. Let \(A\subset X\setminus\{0_X\}\) be nonempty, and let \(\Phi:A\times A\to\R\) be a bifunction. We consider the problem of finding \(\bar y\in A\) such that
\[
\Phi(\bar y,z)\ge0\quad\forall z\in A,
\qquad
\delta_{\cK}(\bar y)\ge\eta_0.
\]
For \(y\in A\), define
\[
E_\Phi(y):=\sup_{z\in A}[-\Phi(y,z)]_+,
\qquad
M_{\alpha,\Phi}(y):=\max\{E_\Phi(y),R_\alpha(y)\}.
\]
Then \(E_\Phi(y)=0\) if and only if \(y\) satisfies the equilibrium inequalities, while the condition \(R_\alpha(y)\le\varepsilon\) implies \(\delta_{\cK}(y)\ge\eta_0-\varepsilon\).

\begin{definition}
\label{def:approximate_radially_feasible_equilibrium}
Let \(\varepsilon\ge0\). A point \(y\in A\) is called an
\(\varepsilon\)-approximate radially feasible equilibrium point if
\[
\Phi(y,z)\ge-\varepsilon\quad\forall z\in A,
\qquad
\delta_{\cK}(y)\ge\eta_0-\varepsilon.
\]
\end{definition}

\begin{proposition}[Equilibrium residual and radial feasibility]
\label{prop:equilibrium_residual_radial_feasibility}
If \(y\in A\) satisfies \(M_{\alpha,\Phi}(y)\le\varepsilon\), then \(y\) is an
\(\varepsilon\)-approximate radially feasible equilibrium point.

Consequently, if \((y_n)\subset A\) satisfies \(M_{\alpha,\Phi}(y_n)\to0\), then
\[
E_\Phi(y_n)\to0
\qquad\text{and}\qquad
\bigl(\eta_0-\delta_{\cK}(y_n)\bigr)_+\to0.
\]
\end{proposition}

\begin{proof}
If \(M_{\alpha,\Phi}(y)\le\varepsilon\), then \(E_\Phi(y)\le\varepsilon\) and
\(R_\alpha(y)\le\varepsilon\). Hence, for every \(z\in A\),
\[
[-\Phi(y,z)]_+\le\varepsilon,
\]
and therefore \(\Phi(y,z)\ge-\varepsilon\). Moreover, by the radial error bound,
\[
\bigl(\eta_0-\delta_{\cK}(y)\bigr)_+
\le R_\alpha(y)
\le\varepsilon.
\]
Thus \(\delta_{\cK}(y)\ge\eta_0-\varepsilon\). The sequential assertion follows
from the same estimates.
\end{proof}

We now specialize this construction to variational inequalities. Let
\(T:A\to X^*\) and define
\[
\Phi_T(y,z):=\langle T(y),z-y\rangle,
\qquad y,z\in A.
\]
The associated equilibrium problem is precisely
\[
\langle T(\bar y),z-\bar y\rangle\ge0
\quad\forall z\in A,
\qquad
\delta_{\cK}(\bar y)\ge\eta_0.
\]
Moreover,
\[
E_{\Phi_T}(y)
=
\sup_{z\in A}\bigl[\langle T(y),y-z\rangle\bigr]_+
=:\Psi_T(y).
\]
Thus the corresponding combined residual is
\[
M_{\alpha,T}(y):=\max\{\Psi_T(y),R_\alpha(y)\},
\qquad y\in A.
\]

\begin{definition}
\label{def:approximate_radially_feasible_vi_solution}
Let \(\varepsilon\ge0\). A point \(y\in A\) is called an
\(\varepsilon\)-approximate radially feasible solution of the variational
inequality if
\[
\langle T(y),z-y\rangle\ge-\varepsilon\quad\forall z\in A,
\qquad
\delta_{\cK}(y)\ge\eta_0-\varepsilon.
\]
\end{definition}

\begin{corollary}[Variational inequality residual and radial feasibility]
\label{cor:vi_residual_radial_feasibility}
If \(y\in A\) satisfies \(M_{\alpha,T}(y)\le\varepsilon\), then \(y\) is an
\(\varepsilon\)-approximate radially feasible solution of the variational
inequality.

Consequently, if \((y_n)\subset A\) satisfies \(M_{\alpha,T}(y_n)\to0\), then
\[
\Psi_T(y_n)\to0
\qquad\text{and}\qquad
\bigl(\eta_0-\delta_{\cK}(y_n)\bigr)_+\to0.
\]
\end{corollary}

\begin{proof}
Apply Proposition~\ref{prop:equilibrium_residual_radial_feasibility} to the
bifunction \(\Phi_T(y,z)=\langle T(y),z-y\rangle\).
\end{proof}

\begin{remark}
\label{rem:combined_residual_interpretation}
The residuals \(M_{\alpha,\Phi}\) and \(M_{\alpha,T}\) are one-sided merit functions. They certify small violation of the equilibrium, respectively
variational inequality, together with small violation of the radial-depth
constraint. They need not characterize exact feasible solutions, since
\(G_\alpha\) may be positive at points satisfying \(\delta_{\cK}(y)\ge\eta_0\).
Their role is a posteriori: small residual implies approximate radial
feasibility.
\end{remark}

\subsubsection{Exact Penalization from the Separation-induced Metric Error Bound}
\label{subsubsec:exact_penalization_from_separation_metric_error_bound}

We now use the metric error bound in
Proposition~\ref{prop:error_bounds_from_separation_residual}\textup{(ii)}. It
shows that the separation residual \(G_\alpha\) controls the distance to its zero
set \(F_\alpha\). This yields an exact penalization principle for optimization
problems constrained by \(F_\alpha\).

\begin{proposition}[Exact penalization on the separation zero set]
\label{prop:global_exact_penalty_from_coradiant_separation}
Fix \(\alpha\in(\alpha_1,\alpha_2)\), and let \(\varphi:X\to\R\) be Lipschitz.
Assume that \(\bar y\in F_\alpha\) solves
\[
\min\{\varphi(y):y\in F_\alpha\}.
\]
If
\[
\rho>\frac{2\Lip(\varphi)}{1-\alpha},
\]
then \(\bar y\) solves
\[
\min\{\varphi(y)+\rho G_\alpha(y):y\in X\}.
\]
\end{proposition}

\begin{proof}
Let \(L:=\Lip(\varphi)\), and fix
\(\rho>2L/(1-\alpha)\). Since \(\bar y\in F_\alpha\), \(G_\alpha(\bar y)=0\).

Let \(y\in X\). If \(y\in F_\alpha\), then \(G_\alpha(y)=0\), and the
optimality of \(\bar y\) on \(F_\alpha\) gives
\[
\varphi(\bar y)+\rho G_\alpha(\bar y)
=
\varphi(\bar y)
\le
\varphi(y)
=
\varphi(y)+\rho G_\alpha(y).
\]

Assume now that \(y\notin F_\alpha\). For every \(z\in F_\alpha\),
\[
\varphi(y)\ge\varphi(z)-L\|y-z\|\ge\varphi(\bar y)-L\|y-z\|.
\]
Taking the infimum over \(z\in F_\alpha\), and using
Proposition~\ref{prop:error_bounds_from_separation_residual}\textup{(ii)}, we get
\[
\varphi(y)
\ge
\varphi(\bar y)-L\dist(y,F_\alpha)
\ge
\varphi(\bar y)-\frac{2L}{1-\alpha}G_\alpha(y).
\]
Hence
\[
\varphi(y)+\rho G_\alpha(y)
\ge
\varphi(\bar y)
+
\left(\rho-\frac{2L}{1-\alpha}\right)G_\alpha(y).
\]
Since \(y\notin F_\alpha\), \(G_\alpha(y)>0\). Therefore
\[
\varphi(y)+\rho G_\alpha(y)>\varphi(\bar y)
=
\varphi(\bar y)+\rho G_\alpha(\bar y).
\]
Thus \(\bar y\) is a global minimizer of \(\varphi+\rho G_\alpha\) on \(X\).
\end{proof}

\begin{remark}
\label{rem:penalty-relational-interpretation}
Proposition~\ref{prop:global_exact_penalty_from_coradiant_separation} is an
exact penalization result for the constraint \(y\in F_\alpha\). Its only
variational input is the metric error bound
\[
\dist(y,F_\alpha)
\le
\frac{2}{1-\alpha}G_\alpha(y),
\qquad y\in X.
\]
The co-radiant structure enters through the construction of \(G_\alpha\): the
separation theorem yields \(\cC\subset F_\alpha\), while
Proposition~\ref{prop:error_bounds_from_separation_residual} gives
\[
F_\alpha\subset
\{y\in X\setminus\{0_X\}:\delta_{\cK}(y)\ge\eta_0\}.
\]
Thus \(F_\alpha\) is a separation-induced feasible region: it contains \(\cC\),
is contained in the prescribed radial-depth region, and admits the exact penalty
\(G_\alpha\).
\end{remark}

\section{Conclusions and Future Research}
\label{sec:conclusions_future_research}

We have introduced regular co-radiant sets and proved a nonlinear separation theorem for this class. The main point is that regularity replaces the existence of a global norm-base by a weaker inner approximation condition. This substantially enlarges the scope of the previous separation theorem for co-radiant sets.

A relevant consequence is that the result applies to the co-radiant sets associated with the main notions of approximate efficiency. Hence the scalarization results previously obtained from co-radiant separation can be used in those standard cases without imposing the restrictive norm-base assumption.

We have also introduced the radial profile of a co-radiant set. This function detects the radial geometry of the set and provides intrinsic criteria for regularity. Moreover, it has an independent scalarization role: it yields scalar tests for approximate efficiency which do not rely on the separation theorem and are valid for arbitrary co-radiant sets.

Finally, we have given a variational reading of the separation residual \(G_\alpha\). The residual controls both the radial-depth violation and the distance to its zero set. These estimates provide certificates of approximate radial feasibility for equilibrium problems and variational inequalities, and yield an exact penalization principle for optimization problems constrained by the separation zero set.

Several questions remain open. One may study sharper criteria for regularity, stability of regular co-radiant sets under perturbations, and extensions to locally convex spaces. Another direction is to develop radial profiles and nonlinear scalarizations adapted to broader classes of radially generated sets.

\subsection*{Funding}

This work was supported by Project PID2021-122126NB-C32, funded by MICIU/AEI/10.13039/501100011033 and by the European Regional Development Fund (ERDF/FEDER), ``A way of making Europe''.

\bibliographystyle{unsrt}
\bibliography{references}

\end{document}